\documentclass[a4paper,12pt]{amsart}

\usepackage[margin=2.5cm]{geometry}

\usepackage{amssymb}
\usepackage{enumerate}
\usepackage{tikz}
\usetikzlibrary{calc}

\title[Balanced simplices]{Balanced simplices}
\author{Jonathan Chappelon}
\thanks{Corresponding Author: Phone: +33-467144166 Email: jonathan.chappelon@um2.fr}
\address{Universit\'{e} Montpellier 2, Institut de Math\'{e}matiques et de Mod\'{e}lisation de Montpellier, Case Courrier 051, Place Eug\`{e}ne Bataillon, 34095 Montpellier Cedex 05, France}
\email{jonathan.chappelon@um2.fr}
\urladdr{http://www.math.univ-montp2.fr/~chappelon/}
\date{September 18, 2014}

\theoremstyle{plain}
\newtheorem{theorem}{Theorem}[section]
\newtheorem{lemma}[theorem]{Lemma}
\newtheorem{corollary}[theorem]{Corollary}
\newtheorem{proposition}[theorem]{Proposition}

\theoremstyle{definition}
\newtheorem{definition}[theorem]{Definition}

\newtheorem{problem}[theorem]{Problem}

\theoremstyle{remark}
\newtheorem{remark}[theorem]{Remark}

\newtheorem{case}{Case}

\renewcommand{\le}{\leqslant}

\renewcommand{\ge}{\geqslant}
\renewcommand{\geq}{\geqslant}

\newcommand{\N}{\mathbb{N}}
\newcommand{\Z}{\mathbb{Z}}
\newcommand{\ZZ}[1]{\Z/{#1}\Z}
\newcommand{\orb}{\mathcal{O}}
\newcommand{\ord}{\mathrm{ord}}
\newcommand{\pord}{\mathrm{pord}}
\newcommand{\AS}{\mathrm{AS}}
\newcommand{\AP}{\mathrm{AP}}
\newcommand{\ArA}{\mathrm{AA}}
\newcommand{\SuS}{\mathrm{SS}}
\newcommand{\m}{\mathfrak{m}}
\newcommand{\lcm}{\mathrm{lcm}}
\newcommand{\PCA}[1]{\mathrm{PCA}_{#1}}

\newcommand{\V}{\mathrm{V}}
\newcommand{\E}{\mathrm{E}}
\newcommand{\F}{\mathrm{F}}
\newcommand{\R}{\mathrm{R}}

\newcommand\frontside[3]{
  \fill[fill=black!75, draw=black,shift={(-#1/3,-#1/3)},shift={(#2,0)},
  shift={(0,#3)}] (0,0) -- (0,-1) -- (-1,-1) --(-1,0)--(0,0);
}

\newcommand\topside[3]{
  \fill[fill=black!30, draw=black,shift={(-#1/3,-#1/3)},shift={(#2,0)},
  shift={(0,#3)}] (0,0) -- (1/3,1/3) -- (-2/3,1/3) --(-1,0)--(0,0);
}

\newcommand\rightside[3]{
  \fill[fill=black!5, draw=black,shift={(-#1/3,-#1/3)},shift={(#2,0)},
  shift={(0,#3)}] (0,0) -- (1/3,1/3) -- (1/3,-2/3) --(0,-1)--(0,0);
}

\newcommand\cube[3]{
  \frontside{#1}{#2}{#3} \topside{#1}{#2}{#3} \rightside{#1}{#2}{#3}
}

\newcommand\tribackleft[4]
{
	\pgfmathparse{#4-1}\let\x\pgfmathresult
	\foreach \i in {0,...,\x}
	{
		\pgfmathparse{#4-1-\i}\let\y\pgfmathresult		
		\foreach \j in {0,...,\y}		
		{
			\pgfmathparse{#1+\i}\let\a\pgfmathresult
			\pgfmathparse{#2+\j}\let\b\pgfmathresult
			\cube{\a}{\b}{#3}
		}
	}	
}

\newcommand\tetratopbackleft[4]{
	\pgfmathparse{#4-1}\let\z\pgfmathresult
	\foreach \k in {0,...,\z}
	{
		\pgfmathparse{#3-#4+1+\k}\let\c\pgfmathresult
		\pgfmathparse{\k +1}\let\t\pgfmathresult
		\tribackleft{#1}{#2}{\c}{\t}
	}
}

\newcommand\tetrabottombackleft[4]{
	\pgfmathparse{#4-1}\let\z\pgfmathresult
	\foreach \k in {0,...,\z}
	{
		\pgfmathparse{#3+\k}\let\c\pgfmathresult
		\pgfmathparse{#4-\k}\let\t\pgfmathresult
		\tribackleft{#1}{#2}{\c}{\t}
	}
}

\newcommand\tribackright[4]
{
	\pgfmathparse{#4-1}\let\x\pgfmathresult
	\foreach \i in {0,...,\x}
	{
		\pgfmathparse{\x -\i}\let\y\pgfmathresult		
		\foreach \j in {0,...,\y}		
		{
			\pgfmathparse{#1+\i}\let\a\pgfmathresult
			\pgfmathparse{#2-\y +\j}\let\b\pgfmathresult
			\cube{\a}{\b}{#3}
		}
	}	
}

\newcommand\tetratopbackright[4]{
	\pgfmathparse{#4-1}\let\z\pgfmathresult
	\foreach \k in {0,...,\z}
	{
		\pgfmathparse{#3-#4+1+\k}\let\c\pgfmathresult
		\pgfmathparse{\k +1}\let\t\pgfmathresult
		\tribackright{#1}{#2}{\c}{\t}
	}
}

\newcommand\tetrabottombackright[4]{
	\pgfmathparse{#4-1}\let\z\pgfmathresult
	\foreach \k in {0,...,\z}
	{
		\pgfmathparse{#3+\k}\let\c\pgfmathresult
		\pgfmathparse{#4-\k}\let\t\pgfmathresult
		\tribackright{#1}{#2}{\c}{\t}
	}
}

\newcommand\trifrontleft[4]
{
	\pgfmathparse{#4-1}\let\x\pgfmathresult
	\foreach \i in {0,...,\x}
	{
		\foreach \j in {0,...,\i}		
		{
			\pgfmathparse{#1 -\x +\i}\let\a\pgfmathresult
			\pgfmathparse{#2+\j}\let\b\pgfmathresult
			\cube{\a}{\b}{#3}
		}
	}	
}

\newcommand\tetratopfrontleft[4]{
	\pgfmathparse{#4-1}\let\z\pgfmathresult
	\foreach \k in {0,...,\z}
	{
		\pgfmathparse{#3-#4+1+\k}\let\c\pgfmathresult
		\pgfmathparse{\k +1}\let\t\pgfmathresult
		\trifrontleft{#1}{#2}{\c}{\t}
	}
}

\newcommand\tetrabottomfrontleft[4]{
	\pgfmathparse{#4-1}\let\z\pgfmathresult
	\foreach \k in {0,...,\z}
	{
		\pgfmathparse{#3+\k}\let\c\pgfmathresult
		\pgfmathparse{#4-\k}\let\t\pgfmathresult
		\trifrontleft{#1}{#2}{\c}{\t}
	}
}

\newcommand\trifrontright[4]
{
	\pgfmathparse{#4-1}\let\x\pgfmathresult
	\foreach \i in {0,...,\x}
	{
		\foreach \j in {0,...,\i}		
		{
			\pgfmathparse{#1-\x +\i}\let\a\pgfmathresult
			\pgfmathparse{#2-\i +\j}\let\b\pgfmathresult
			\cube{\a}{\b}{#3}
		}
	}	
}

\newcommand\tetratopfrontright[4]{
	\pgfmathparse{#4-1}\let\z\pgfmathresult
	\foreach \k in {0,...,\z}
	{
		\pgfmathparse{#3-#4+1+\k}\let\c\pgfmathresult
		\pgfmathparse{\k +1}\let\t\pgfmathresult
		\trifrontright{#1}{#2}{\c}{\t}
	}
}

\newcommand\tetrabottomfrontright[4]{
	\pgfmathparse{#4-1}\let\z\pgfmathresult
	\foreach \k in {0,...,\z}
	{
		\pgfmathparse{#3+\k}\let\c\pgfmathresult
		\pgfmathparse{#4-\k}\let\t\pgfmathresult
		\trifrontright{#1}{#2}{\c}{\t}
	}
}

\begin{document}

\begin{abstract}
An additive cellular automaton is a linear map on the set of infinite multidimensional arrays of elements in a finite cyclic group $\mathbb{Z}/m\mathbb{Z}$. In this paper, we consider simplices appearing in the orbits generated from arithmetic arrays by additive cellular automata. We prove that they are a source of balanced simplices, that are simplices containing all the elements of $\mathbb{Z}/m\mathbb{Z}$ with the same multiplicity. For any additive cellular automaton of dimension $1$ or higher, the existence of infinitely many balanced simplices of $\mathbb{Z}/m\mathbb{Z}$ appearing in such orbits is shown, and this, for an infinite number of values $m$. The special case of the Pascal cellular automata, the cellular automata generating the Pascal simplices, that are a generalization of the Pascal triangle into arbitrary dimension, is studied in detail.
\end{abstract}

\maketitle

\section{Introduction}

Let $m$ and $s$ be two positive integers. A \textit{Steinhaus triangle} modulo $m$ of size $s$ is a triangle, with $s$ rows, whose first row is composed of $s$ elements in $\ZZ{m}$ and satisfying the following local rule: each entry of the triangle (not on the first row) is the sum modulo $m$ of the two directly above it. A \textit{generalized Pascal triangle} modulo $m$ of size $s$ is defined like a Steinhaus triangle, but with the other possible orientation. More precisely, it is a triangle, with $s$ rows, whose last row is composed of $s$ elements in $\ZZ{m}$ and satisfying the same local rule. Examples of these triangles are depicted in Figure~\ref{fig0}.

\begin{figure}[!h]
\begin{center}
\begin{tikzpicture}[scale=0.25]

\pgfmathparse{sqrt(3)}\let\x\pgfmathresult

\node at (0,0) {$2$};
\node at (2,0) {$3$};
\node at (4,0) {$3$};
\node at (6,0) {$0$};
\node at (8,0) {$4$};
\node at (10,0) {$4$};
\node at (12,0) {$5$};

\node at (1,-\x) {$5$};
\node at (3,-\x) {$6$};
\node at (5,-\x) {$3$};
\node at (7,-\x) {$4$};
\node at (9,-\x) {$1$};
\node at (11,-\x) {$2$};

\node at (2,-2*\x) {$4$};
\node at (4,-2*\x) {$2$};
\node at (6,-2*\x) {$0$};
\node at (8,-2*\x) {$5$};
\node at (10,-2*\x) {$3$};

\node at (3,-3*\x) {$6$};
\node at (5,-3*\x) {$2$};
\node at (7,-3*\x) {$5$};
\node at (9,-3*\x) {$1$};

\node at (4,-4*\x) {$1$};
\node at (6,-4*\x) {$0$};
\node at (8,-4*\x) {$6$};

\node at (5,-5*\x) {$1$};
\node at (7,-5*\x) {$6$};

\node at (6,-6*\x) {$0$};

\draw (-1.5,0.5*\x) -- (13.5,0.5*\x) -- (6,-7*\x) -- (-1.5,0.5*\x);

\end{tikzpicture}\hspace{1.5cm}
\begin{tikzpicture}[scale=0.25]
\pgfmathparse{sqrt(3)}\let\x\pgfmathresult

\node at (0,0) {$0$};

\node at (-1,-\x){$1$};
\node at (1,-\x){$6$};

\node at (-2,-2*\x){$2$};
\node at (0,-2*\x){$0$};
\node at (2,-2*\x){$5$};

\node at (-3,-3*\x){$1$};
\node at (-1,-3*\x){$2$};
\node at (1,-3*\x){$5$};
\node at (3,-3*\x){$6$};

\node at (-4,-4*\x){$5$};
\node at (-2,-4*\x){$3$};
\node at (0,-4*\x){$0$};
\node at (2,-4*\x){$4$};
\node at (4,-4*\x){$2$};

\node at (-5,-5*\x){$3$};
\node at (-3,-5*\x){$1$};
\node at (-1,-5*\x){$3$};
\node at (1,-5*\x){$4$};
\node at (3,-5*\x){$6$};
\node at (5,-5*\x){$4$};

\draw (0,\x) -- (-6.5,-5.5*\x) -- (6.5,-5.5*\x) -- (0,\x);

\end{tikzpicture}
\end{center}
\caption{A Steinhaus triangle and a generalized Pascal triangle modulo $7$}\label{fig0}
\end{figure}
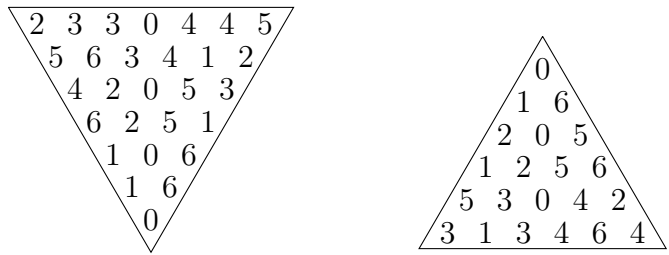

A typical problem on such triangles is to determine sizes for which there exists a \textit{balanced} triangle modulo $m$, that is a triangle containing all the elements of $\ZZ{m}$ with the same multiplicity. For instance, the triangles in Figure~\ref{fig0} are balanced modulo $7$.

The purpose of this paper is to study balanced simplices generated by additive cellular automata, that are a generalization of Steinhaus triangles and generalized Pascal triangles into arbitrary dimension and satisfying a local rule that is not necessarily the sum modulo $m$, but any linear map.

\subsection{Known results on Steinhaus triangles and generalized Pascal triangles}

The name of Steinhaus triangle is due to Hugo Steinhaus himself in \cite{Steinhaus1964}, where he proposed this construction in the binary case $m=2$. He posed the following problem, as an unsolved problem.

\begin{problem}[Steinhaus \cite{Steinhaus1964}]\label{probstein}
Does there exist, for all positive integers $s$ such that $s\equiv 0$ or $3\bmod{4}$, a Steinhaus triangle of size $s$ in $\ZZ{2}$ containing as many $0$'s as $1$'s?
\end{problem}
Remark that the condition on the size $s$ of a balanced Steinhaus triangle in $\ZZ{2}$ is obviously a necessary condition because the number of elements of a such triangle, that is $\binom{s+1}{2}$, is even if and only if $s\equiv 0$ or $3\bmod{4}$. A positive solution to this problem appeared in the literature for the first time in \cite{Harborth1972}, where the author gave, for every $s\equiv0$ or $3\bmod{4}$, an explicit construction of a balanced Steinhaus triangle of size $s$ in $\ZZ{2}$. More recently, several other constructions of balanced binary Steinhaus triangles have been obtained by considering sequences with additional properties such as strongly balanced \cite{Eliahou2004}, symmetric and antisymmetric \cite{Eliahou2005}, or zero-sum sequences \cite{Eliahou2007}.

The number of occurrences of $1$ appearing in binary Steinhaus triangles and in binary generalized Pascal triangles are studied in \cite{Chang1983} and in \cite{Harborth2005}, respectively. Such binary triangles having certain geometric properties are studied in \cite{Barbe2000,Brunat2011}. A binary Steinhaus triangle can also be considered as the upper triangular part of the adjacency matrix of a finite graph. These undirected graphs are called Steinhaus graphs in \cite{Molluzzo1978}. A classical problem on Steinhaus graphs is to study those having certain graphical properties such as bipartition \cite{Chang1999,Dymacek1986,Dymacek1995}, planarity \cite{Dymacek2000} or regularity \cite{Augier2008,Bailey1988,Chappelon2009,Dymacek1979}. A survey on Steinhaus graphs can be found in \cite{Dymacek1999}.

Problem~\ref{probstein} on balanced binary Steinhaus triangles was generalized for any positive integers $m$ by Molluzzo in \cite{Molluzzo1978}.

\begin{problem}[Molluzzo \cite{Molluzzo1978}]\label{probmolluzzo}
Let $m$ be a positive integer. Does there exist, for all positive integers $s$ such that $\binom{s+1}{2}$ is divisible by $m$, a Steinhaus triangle of size $s$ containing all the elements of $\ZZ{m}$ with the same multiplicity?
\end{problem}

Since then, this problem has been positively solved, by constructive approaches, for small values of $m$: for $m\in\{3,5\}$ in \cite{Bartsch1985}, for $m\in\{3,5,7\}$ in \cite{Chappelon2008a}, for $m=4$ in \cite{Chappelon2012a}. First counter-examples appeared in \cite{Chappelon2008a}, where the author proved that there does not exist balanced Steinhaus triangles of size $5$ in $\ZZ{15}$ and of size $6$ in $\ZZ{21}$. Nevertheless, this problem can be positively answered for an infinite number of values $m$. Indeed, as showed in \cite{Chappelon2008a,Chappelon2008}, there exist balanced Steinhaus triangles, for all the possible sizes, in the case where $m$ is a power of $3$. More precisely, the author obtained the following result.

\begin{theorem}[Chappelon \cite{Chappelon2008a,Chappelon2008}]\label{chap1}
Let $m$ be an odd number and let $a$ and $d$ be in $\ZZ{m}$ such that $d$ is invertible. The Steinhaus triangle, of size $s$, whose first row is the arithmetic progression $(a,a+d,a+2d,\ldots,a+(s-1)d)$ is balanced in $\ZZ{m}$, for all $s\equiv 0\ \text{or}\ -1\pmod{\ord_{m}(2^m)m}$, where $\ord_{m}(2^m)$ is the multiplicative order of $2^m$ modulo $m$, i.e., the order of $2^m$ in the group of invertibles $(\ZZ{m})^*$.
\end{theorem}

In particular, when $m=3^k$, the equality $\ord_{m}(2^m)=2$ can be easily proved by induction on $k\geq 1$. It follows, from Theorem~\ref{chap1}, that there exist balanced Steinhaus triangles of size $s$ in $\ZZ{3^k}$ for all $s\equiv 0$ or $-1\bmod{2.3^k}$. This result can be refined by considering Steinhaus triangles whose first row has the additional property to be antisymmetric. Thus, the author obtained a positive answer to the Molluzzo problem for all $m=3^k$. Even if the Molluzzo problem is not completely solved for the other odd values of $m$, we know from Theorem~\ref{chap1} that there exist infinitely many balanced Steinhaus triangles in every $\ZZ{m}$ with $m$ odd. This weak version of the Molluzzo problem was posed in \cite{Chappelon2012a}.

\begin{problem}[Weak Molluzzo problem]\label{weakmolluzzo}
Let $m$ be a positive integer. Do there exist infinitely many balanced Steinhaus triangles of $\ZZ{m}$?
\end{problem}

Problem~\ref{weakmolluzzo} is thus solved for the odd numbers $m$. For the even values, the cases $m=2$ and $m=4$ come from the solutions of Problem~\ref{probmolluzzo} and a solution will appear in \cite{Eliahou2014} for $m\in\{6,8,10\}$. This problem is completely open for the even numbers $m\ge12$. Similar problems of determining the existence of balanced structures can be considered for other shapes. In \cite{Chappelon2011}, the author not only considers balanced Steinhaus triangles, but also balanced generalized Pascal triangles, trapezoids or lozenges. In particular, for Steinhaus triangles and generalized Pascal triangles, the following result is proved.

\begin{theorem}[Chappelon, \cite{Chappelon2011}]\label{chap2}
Let $m$ be an odd number. For all $s\equiv 0\bmod{m}$ and for all $s\equiv -1 \bmod{3m}$, there exist Steinhaus triangles and generalized Pascal triangles of size $s$, which are balanced in $\ZZ{m}$.
\end{theorem}

In fact, the author proves that all the balanced triangles announced in Theorem~\ref{chap2} are obtained from the following sequence
$$
(\ldots,0,-1,1,1,-3,2,2,-5,3,3,-7,4,4,-9,5,5,\ldots),
$$
which is an interlacing of three arithmetic progressions, and this, for every odd number $m$. The main ingredient in the proof of this result is an elementary object that the author called a doubly arithmetic triangle in \cite{Chappelon2011} and that we simply call an arithmetic triangle here in this paper. An arithmetic triangle with common differences $d_1$ and $d_2$ is a triangle of elements in $\ZZ{m}$ whose rows and columns are arithmetic progressions with the same common differences, $d_1$ for the rows and $d_2$ for the columns. The interest of this structure is that it is very often balanced.

\subsection{Simplices generated by additive cellular automata}

Let us now define the generalization of Steinhaus triangles and generalized Pascal triangles that we consider here. Let $n$ and $m$ be positive integers. Throughout this paper, $n$ will denote the dimension of the objects studied and $m$ the order of the finite cyclic group $\ZZ{m}$. For any integers $a$ and $b$ such that $a<b$, we let $[a,b]$ denote the set of the integers between $a$ and $b$, that is, $[a,b]:=\{a,a+1,\ldots,b\}$ and $[a,b]^{n}$ the Cartesian product of $n$ copies of $[a,b]$. For any $n$-tuple of elements $u$, we let $u_i$ denote its $i$th component for all $i\in[1,n]$, that is, $u=(u_1,\ldots,u_n)$. For two $n$-tuples $u$ and $v$ and an integer $\lambda$, we consider the sum $u+v:=(u_1+v_1,\ldots,u_n+v_n)$, the product $u\cdot v := (u_1v_1,\ldots,u_n v_n)$ and the scalar product $\lambda u:=(\lambda u_1,\ldots,\lambda u_n)$.

\begin{definition}[ACA]
Let $r$ be a non-negative integer and let $W=(w_j)_{j\in[-r,r]^n}$ be an $n$-dimensional array of integers of size $(2r+1)^{n}$. The \textit{additive cellular automaton} (ACA for short) over $\ZZ{m}$ associated with $W$ is the map $\partial$ which assigns, to every $n$-dimensional infinite array of $\ZZ{m}$, a new array by a linear transformation whose coefficients are those of $W$. More precisely, the map $\partial$ is defined by
$$
\partial\left( (a_i)_{i\in\Z^n} \right) = \left(\sum_{j\in[-r,r]^n}{w_ja_{i+j}}\right)_{i\in\Z^n},
$$
for all arrays $(a_i)_{i\in\Z^n}$ of elements in $\ZZ{m}$. We say that $\partial$ is of \textit{dimension} $n\ge1$ and of \textit{weight} $W$ with \textit{radius} $r\ge0$.
\end{definition}

\begin{definition}[Orbit]
Let $A=(a_i)_{i\in\Z^n}$ be an infinite array of $\ZZ{m}$ of dimension $n$. The \textit{orbit} $\orb(A)$ generated from $A$ by the ACA $\partial$ is the collection of all the $n$-dimensional arrays obtained from $A$ by successive applications of $\partial$, that is,
$$
\orb(A) := \left\{ \partial^{j}(A)\ \middle|\ j\in\N\right\},
$$
where $\partial^{j}$ is recursively defined by $\partial^{j}(A)=\partial(\partial^{j-1}(A))$ for all $j\ge1$ and $\partial^{0}(A)=A$. The orbit $\orb(A)$ can also be seen as the $(n+1)$-dimensional array $\left(a_{i,j}\right)_{(i,j)\in\Z^{n}\times\N}$ of $\ZZ{m}$ whose $j$th row $R_j:=\left(a_{i,j}\right)_{i\in\Z^n}$ corresponds to $\partial^{j}(A)$, for all $j\in\N$.
\end{definition}

\begin{definition}[Simplices]
Let $A=\left( a_i \right)_{i\in\Z^{n}}$ be an infinite array of $\ZZ{m}$ of dimension $n$. Let $\varepsilon\in\{-1,1\}^{n}$ and let $s$ be a positive integer. The \textit{simplex} of size $s$, with orientation $\varepsilon$ and whose principal vertex is at the coordinates $j\in\Z^{n}$ in $A$, is the multiset of $\ZZ{m}$ defined and denoted by
$$
\triangle(j,\varepsilon,s) := \left\{ a_{j+\varepsilon \cdot k}\ \middle|\ k\in\N^{n}\ \text{such\ that}\ k_1+\cdots+k_n \le s-1 \right\}.
$$
For $n=2$ and $n=3$, it is called a triangle and a tetrahedron, respectively.
\end{definition}

\begin{figure}
\begin{center}
\begin{tikzpicture}[scale=0.5]
\draw[fill=black!40,draw=white] (2,9) -- (2,14) -- (7,14) -- (7,13) -- (6,13) -- (6,12) -- (5,12) -- (5,11) -- (4,11) -- (4,10) -- (3,10) -- (3,9) -- (2,9);
\draw[fill=black!40,draw=white] (2,2) -- (2,7) -- (3,7)  -- (3,6) -- (4,6) -- (4,5) -- (5,5) -- (5,4) -- (6,4) -- (6,3) -- (7,3) -- (7,2) -- (2,2);
\draw[fill=black!40,draw=white] (9,2) -- (14,2) -- (14,7) -- (13,7) -- (13,6) -- (12,6) -- (12,5) -- (11,5) -- (11,4) -- (10,4) -- (10,3) -- (9,3) -- (9,2);
\draw[fill=black!40,draw=white] (9,14) -- (14,14) -- (14,9) -- (13,9) -- (13,10) -- (12,10) -- (12,11) -- (11,11) -- (11,12) -- (10,12) -- (10,13) -- (9,13) -- (9,14);
\draw (0,0) -- (0,16) -- (16,16) -- (16,0) -- (0,0);
\draw (0,15) -- (16,15);
\draw (0,14) -- (16,14);
\draw (0,13) -- (16,13);
\draw (0,12) -- (16,12);
\draw (0,11) -- (16,11);
\draw (0,10) -- (16,10);
\draw (0,9) -- (16,9);
\draw (0,8) -- (16,8);
\draw (0,7) -- (16,7);
\draw (0,6) -- (16,6);
\draw (0,5) -- (16,5);
\draw (0,4) -- (16,4);
\draw (0,3) -- (16,3);
\draw (0,2) -- (16,2);
\draw (0,1) -- (16,1);
\draw (1,0) -- (1,16);
\draw (2,0) -- (2,16);
\draw (3,0) -- (3,16);
\draw (4,0) -- (4,16);
\draw (5,0) -- (5,16);
\draw (6,0) -- (6,16);
\draw (7,0) -- (7,16);
\draw (8,0) -- (8,16);
\draw (9,0) -- (9,16);
\draw (10,0) -- (10,16);
\draw (11,0) -- (11,16);
\draw (12,0) -- (12,16);
\draw (13,0) -- (13,16);
\draw (14,0) -- (14,16);
\draw (15,0) -- (15,16);

\node at (0.5,15.5) {\tiny $a_{0,0}$};
\node at (1.5,15.5) {\tiny $a_{1,0}$};
\node at (2.5,15.5) {$2$};
\node at (3.5,15.5) {$3$};
\node at (4.5,15.5) {$4$};
\node at (5.5,15.5) {$0$};
\node at (6.5,15.5) {$1$};
\node at (7.5,15.5) {$2$};
\node at (8.5,15.5) {$3$};
\node at (9.5,15.5) {$4$};
\node at (10.5,15.5) {$0$};
\node at (11.5,15.5) {$1$};
\node at (12.5,15.5) {$2$};
\node at (13.5,15.5) {$3$};
\node at (14.5,15.5) {$4$};
\node at (15.5,15.5) {$0$};

\node at (0.5,14.5) {\tiny $a_{0,1}$};
\node at (1.5,14.5) {\tiny $a_{1,1}$};
\node at (2.5,14.5) {$2$};
\node at (3.5,14.5) {$1$};
\node at (4.5,14.5) {$0$};
\node at (5.5,14.5) {$4$};
\node at (6.5,14.5) {$3$};
\node at (7.5,14.5) {$2$};
\node at (8.5,14.5) {$1$};
\node at (9.5,14.5) {$0$};
\node at (10.5,14.5) {$4$};
\node at (11.5,14.5) {$3$};
\node at (12.5,14.5) {$2$};
\node at (13.5,14.5) {$1$};
\node at (14.5,14.5) {$0$};
\node at (15.5,14.5) {$4$};

\node at (0.5,13.5) {$2$};
\node at (1.5,13.5) {$3$};
\node at (2.5,13.5) {$4$};
\node at (3.5,13.5) {$0$};
\node at (4.5,13.5) {$1$};
\node at (5.5,13.5) {$2$};
\node at (6.5,13.5) {$3$};
\node at (7.5,13.5) {$4$};
\node at (8.5,13.5) {$0$};
\node at (9.5,13.5) {$1$};
\node at (10.5,13.5) {$2$};
\node at (11.5,13.5) {$3$};
\node at (12.5,13.5) {$4$};
\node at (13.5,13.5) {$0$};
\node at (14.5,13.5) {$1$};
\node at (15.5,13.5) {$2$};

\node at (0.5,12.5) {$2$};
\node at (1.5,12.5) {$1$};
\node at (2.5,12.5) {$0$};
\node at (3.5,12.5) {$4$};
\node at (4.5,12.5) {$3$};
\node at (5.5,12.5) {$2$};
\node at (6.5,12.5) {$1$};
\node at (7.5,12.5) {$0$};
\node at (8.5,12.5) {$4$};
\node at (9.5,12.5) {$3$};
\node at (10.5,12.5) {$2$};
\node at (11.5,12.5) {$1$};
\node at (12.5,12.5) {$0$};
\node at (13.5,12.5) {$4$};
\node at (14.5,12.5) {$3$};
\node at (15.5,12.5) {$2$};

\node at (0.5,11.5) {$4$};
\node at (1.5,11.5) {$0$};
\node at (2.5,11.5) {$1$};
\node at (3.5,11.5) {$2$};
\node at (4.5,11.5) {$3$};
\node at (5.5,11.5) {$4$};
\node at (6.5,11.5) {$0$};
\node at (7.5,11.5) {$1$};
\node at (8.5,11.5) {$2$};
\node at (9.5,11.5) {$3$};
\node at (10.5,11.5) {$4$};
\node at (11.5,11.5) {$0$};
\node at (12.5,11.5) {$1$};
\node at (13.5,11.5) {$2$};
\node at (14.5,11.5) {$3$};
\node at (15.5,11.5) {$4$};

\node at (0.5,10.5) {$0$};
\node at (1.5,10.5) {$4$};
\node at (2.5,10.5) {$3$};
\node at (3.5,10.5) {$2$};
\node at (4.5,10.5) {$1$};
\node at (5.5,10.5) {$0$};
\node at (6.5,10.5) {$4$};
\node at (7.5,10.5) {$3$};
\node at (8.5,10.5) {$2$};
\node at (9.5,10.5) {$1$};
\node at (10.5,10.5) {$0$};
\node at (11.5,10.5) {$4$};
\node at (12.5,10.5) {$3$};
\node at (13.5,10.5) {$2$};
\node at (14.5,10.5) {$1$};
\node at (15.5,10.5) {$0$};

\node at (0.5,9.5) {$1$};
\node at (1.5,9.5) {$2$};
\node at (2.5,9.5) {$3$};
\node at (3.5,9.5) {$4$};
\node at (4.5,9.5) {$0$};
\node at (5.5,9.5) {$1$};
\node at (6.5,9.5) {$2$};
\node at (7.5,9.5) {$3$};
\node at (8.5,9.5) {$4$};
\node at (9.5,9.5) {$0$};
\node at (10.5,9.5) {$1$};
\node at (11.5,9.5) {$2$};
\node at (12.5,9.5) {$3$};
\node at (13.5,9.5) {$4$};
\node at (14.5,9.5) {$0$};
\node at (15.5,9.5) {$1$};

\node at (0.5,8.5) {$4$};
\node at (1.5,8.5) {$3$};
\node at (2.5,8.5) {$2$};
\node at (3.5,8.5) {$1$};
\node at (4.5,8.5) {$0$};
\node at (5.5,8.5) {$4$};
\node at (6.5,8.5) {$3$};
\node at (7.5,8.5) {$2$};
\node at (8.5,8.5) {$1$};
\node at (9.5,8.5) {$0$};
\node at (10.5,8.5) {$4$};
\node at (11.5,8.5) {$3$};
\node at (12.5,8.5) {$2$};
\node at (13.5,8.5) {$1$};
\node at (14.5,8.5) {$0$};
\node at (15.5,8.5) {$4$};

\node at (0.5,7.5) {$2$};
\node at (1.5,7.5) {$3$};
\node at (2.5,7.5) {$4$};
\node at (3.5,7.5) {$0$};
\node at (4.5,7.5) {$1$};
\node at (5.5,7.5) {$2$};
\node at (6.5,7.5) {$3$};
\node at (7.5,7.5) {$4$};
\node at (8.5,7.5) {$0$};
\node at (9.5,7.5) {$1$};
\node at (10.5,7.5) {$2$};
\node at (11.5,7.5) {$3$};
\node at (12.5,7.5) {$4$};
\node at (13.5,7.5) {$0$};
\node at (14.5,7.5) {$1$};
\node at (15.5,7.5) {$2$};

\node at (0.5,6.5) {$2$};
\node at (1.5,6.5) {$1$};
\node at (2.5,6.5) {$0$};
\node at (3.5,6.5) {$4$};
\node at (4.5,6.5) {$3$};
\node at (5.5,6.5) {$2$};
\node at (6.5,6.5) {$1$};
\node at (7.5,6.5) {$0$};
\node at (8.5,6.5) {$4$};
\node at (9.5,6.5) {$3$};
\node at (10.5,6.5) {$2$};
\node at (11.5,6.5) {$1$};
\node at (12.5,6.5) {$0$};
\node at (13.5,6.5) {$4$};
\node at (14.5,6.5) {$3$};
\node at (15.5,6.5) {$2$};

\node at (0.5,5.5) {$4$};
\node at (1.5,5.5) {$0$};
\node at (2.5,5.5) {$1$};
\node at (3.5,5.5) {$2$};
\node at (4.5,5.5) {$3$};
\node at (5.5,5.5) {$4$};
\node at (6.5,5.5) {$0$};
\node at (7.5,5.5) {$1$};
\node at (8.5,5.5) {$2$};
\node at (9.5,5.5) {$3$};
\node at (10.5,5.5) {$4$};
\node at (11.5,5.5) {$0$};
\node at (12.5,5.5) {$1$};
\node at (13.5,5.5) {$2$};
\node at (14.5,5.5) {$3$};
\node at (15.5,5.5) {$4$};

\node at (0.5,4.5) {$0$};
\node at (1.5,4.5) {$4$};
\node at (2.5,4.5) {$3$};
\node at (3.5,4.5) {$2$};
\node at (4.5,4.5) {$1$};
\node at (5.5,4.5) {$0$};
\node at (6.5,4.5) {$4$};
\node at (7.5,4.5) {$3$};
\node at (8.5,4.5) {$2$};
\node at (9.5,4.5) {$1$};
\node at (10.5,4.5) {$0$};
\node at (11.5,4.5) {$4$};
\node at (12.5,4.5) {$3$};
\node at (13.5,4.5) {$2$};
\node at (14.5,4.5) {$1$};
\node at (15.5,4.5) {$0$};

\node at (0.5,3.5) {$1$};
\node at (1.5,3.5) {$2$};
\node at (2.5,3.5) {$3$};
\node at (3.5,3.5) {$4$};
\node at (4.5,3.5) {$0$};
\node at (5.5,3.5) {$1$};
\node at (6.5,3.5) {$2$};
\node at (7.5,3.5) {$3$};
\node at (8.5,3.5) {$4$};
\node at (9.5,3.5) {$0$};
\node at (10.5,3.5) {$1$};
\node at (11.5,3.5) {$2$};
\node at (12.5,3.5) {$3$};
\node at (13.5,3.5) {$4$};
\node at (14.5,3.5) {$0$};
\node at (15.5,3.5) {$1$};

\node at (0.5,2.5) {$4$};
\node at (1.5,2.5) {$3$};
\node at (2.5,2.5) {$2$};
\node at (3.5,2.5) {$1$};
\node at (4.5,2.5) {$0$};
\node at (5.5,2.5) {$4$};
\node at (6.5,2.5) {$3$};
\node at (7.5,2.5) {$2$};
\node at (8.5,2.5) {$1$};
\node at (9.5,2.5) {$0$};
\node at (10.5,2.5) {$4$};
\node at (11.5,2.5) {$3$};
\node at (12.5,2.5) {$2$};
\node at (13.5,2.5) {$1$};
\node at (14.5,2.5) {$0$};
\node at (15.5,2.5) {$4$};

\node at (0.5,1.5) {$2$};
\node at (1.5,1.5) {$3$};
\node at (2.5,1.5) {$4$};
\node at (3.5,1.5) {$0$};
\node at (4.5,1.5) {$1$};
\node at (5.5,1.5) {$2$};
\node at (6.5,1.5) {$3$};
\node at (7.5,1.5) {$4$};
\node at (8.5,1.5) {$0$};
\node at (9.5,1.5) {$1$};
\node at (10.5,1.5) {$2$};
\node at (11.5,1.5) {$3$};
\node at (12.5,1.5) {$4$};
\node at (13.5,1.5) {$0$};
\node at (14.5,1.5) {$1$};
\node at (15.5,1.5) {$2$};

\node at (0.5,0.5) {$2$};
\node at (1.5,0.5) {$1$};
\node at (2.5,0.5) {$0$};
\node at (3.5,0.5) {$4$};
\node at (4.5,0.5) {$3$};
\node at (5.5,0.5) {$2$};
\node at (6.5,0.5) {$1$};
\node at (7.5,0.5) {$0$};
\node at (8.5,0.5) {$4$};
\node at (9.5,0.5) {$3$};
\node at (10.5,0.5) {$2$};
\node at (11.5,0.5) {$1$};
\node at (12.5,0.5) {$0$};
\node at (13.5,0.5) {$4$};
\node at (14.5,0.5) {$3$};
\node at (15.5,0.5) {$2$};
\end{tikzpicture}
\end{center}
\caption{\label{fig1}Example of triangles $\triangle((2,2),++,5)$, $\triangle((2,13),+-,5)$, $\triangle((13,2),-+,5)$ and $\triangle((13,13),--,5)$ appearing in an orbit $\orb(A)=(a_{i,j})_{(i,j)\in\Z\times\N}$ of $\ZZ{5}$ generated by the ACA of weight $W=(2,1,1)$.}
\end{figure}

\begin{figure}
\begin{center}
\begin{tikzpicture}[scale=0.4]

\tetrabottombackleft{0}{0}{0}{4}
\tetrabottombackright{0}{12}{0}{4}
\tetrabottomfrontleft{12}{0}{0}{4}
\tetrabottomfrontright{12}{12}{0}{4}

\tetratopbackleft{0}{0}{12}{4}
\tetratopbackright{0}{12}{12}{4}
\tetratopfrontleft{12}{0}{12}{4}
\tetratopfrontright{12}{12}{12}{4}

\end{tikzpicture}
\caption{The eight possible orientations of a tetrahedron}\label{fig2}
\end{center}
\end{figure}
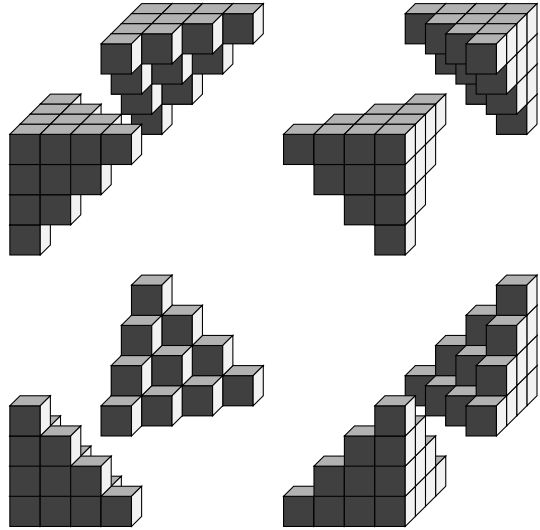

In this paper, we mainly consider simplices of dimension $n$ appearing in the orbit generated from an infinite array of $\ZZ{m}$ by an ACA of dimension $n-1$. Examples of triangles, for the four possible orientations in dimension $n=2$, are depicted in Figure~\ref{fig1}. In Figure~\ref{fig2}, for dimension $n=3$, the eight possible orientations of a tetrahedra are represented.

Let $M$ be a \textit{multiset} of $\ZZ{m}$, that is a set where each element of $\ZZ{m}$ can appear more than once. The \textit{multiplicity function} associated with $M$ is the integer-valued function $\m_{M} : \ZZ{m} \longmapsto \N$, which assigns to each element of $\ZZ{m}$ its multiplicity in $M$. The \textit{cardinality} of $M$, denoted by $|M|$, is the number of elements of $M$, counted with multiplicity, that is, $|M|=\sum_{x\in\ZZ{m}}\m_M(x)$.

\begin{definition}[Balanced multisets of $\ZZ{m}$]
A multiset of $\ZZ{m}$ is said to be \textit{balanced} if it contains all the elements of $\ZZ{m}$ with the same multiplicity, i.e., if its associated multiplicity function $\m_M$ is constant on $\ZZ{m}$, equal to $\frac{1}{m}|M|$.
\end{definition}

The goal of this paper is to prove the existence of balanced simplices appearing in certain orbits generated by ACA. Sufficient conditions for obtaining this result will be detailed throughout this paper. This notion of balanced simplices generated by ACA essentially appears in the literature in the case of the Pascal cellular automaton of dimension 1.

\begin{definition}[$\PCA{n}$]
The \textit{Pascal cellular automaton} of dimension $n$ is the ACA  of radius $r=1$ and whose weight array $W=(w_i)_{i\in[-1,1]^{n}}$ is defined by
$$
w_{i} = \left\{\begin{array}{cl}
1 & \text{if}\ i\in\left\{ 0_{\Z^{n}} , -e_1 , -e_2 , \ldots , -e_n \right\}, \\
0 & \text{otherwise},
\end{array}\right.
$$
where $(e_1,e_2,\ldots,e_n)$ is the canonical basis of the vector space $\Z^n$. It is denoted by $\PCA{n}$.
\end{definition}

For instance, $W=(1,1,0)$ for $\PCA{1}$ and $W=\left(\begin{array}{ccc}
 0 & 0 & 0 \\
 1 & 1 & 0 \\
 0 & 1 & 0 \\
\end{array}\right)$ for $\PCA{2}$.

\begin{remark}
Let $A=(a_{i})_{i\in\Z^{n-1}}$ be the $(n-1)$-dimensional array of $\ZZ{m}$ defined by $a_{i}=1$ for $i=0_{\Z^{n}}$ and $a_i=0$ otherwise. If $(a_{i,j})_{(i,j)\in\Z^{n-1}\times\N}$ is the orbit $\orb(A)$ generated by the $\PCA{n-1}$, then $a_{i,j}$ is the multinomial coefficient
$$
a_{i,j} = \binom{j}{i_1,\ldots,i_{n-1},j-\sum_{k=1}^{n-1}i_{k}} = \frac{j!}{i_1!\cdots i_{n-1}!\left(j-\sum_{k=1}^{n-1}i_{k}\right)!}
$$
for all $i\in\N^{n-1}$ such that $i_1+\cdots+i_{n-1}\le j$, and $a_{i,j}=0$ otherwise. Thus, we retrieve the coefficients of the Pascal $n$-simplex modulo $m$. This is the reason why this specific ACA is called the Pascal cellular automaton.
\end{remark}

Let $A=(a_i)_{i\in\Z}$ be a doubly infinite sequence of $\ZZ{m}$. We consider the orbit generated from $A$ by $\PCA{1}$, i.e., the infinite array $\orb(A)=(a_{i,j})_{(i,j)\in\Z\times\N}$ defined by $a_{i,0} = a_{i}$ for all $i\in\Z$ and
$$
a_{i,j} = a_{i-1,j-1} + a_{i,j-1}
$$
for all $(i,j)\in\Z\times\N^{*}$. Then, the $(-+)$-triangles and the $(+-)$-triangles appearing in a such orbit correspond with the Steinhaus triangles and the generalized Pascal triangles modulo $m$, respectively.

\subsection{Contents}

As already seen before, the proof of Theorem~\ref{chap2} in \cite{Chappelon2011} is based on the elementary object that is the arithmetic triangle and its main interest is that it is very often balanced. In this paper, we consider a generalization in higher dimensions of arithmetic triangles.

\begin{definition}[Arithmetic arrays and simplices]\label{defarith}
Let $n$ and $m$ be positive integers. Let $A=(a_{i})_{i\in\Z^{n}}$ be an array of $\ZZ{m}$. The array $A$ is said to be \textit{arithmetic} with first element $a$ and with common difference $d=(d_1,\ldots,d_n)\in(\ZZ{m})^n$ if
$$
a_{i} = a+i_1d_1+\cdots+i_nd_n,
$$
for all $i=(i_1,\ldots,i_n)\in\Z^n$. The arithmetic array with first element $a\in\ZZ{m}$ and with common difference $d\in(\ZZ{m})^{n}$ is denoted by $\ArA(a,d)$. The \textit{arithmetic simplex} of size $s$, with first element $a\in\ZZ{m}$ and with common difference $d=(d_1,\ldots,d_n)\in(\ZZ{m})^{n}$, is the simplex $\triangle(0_{\Z^{n}},+\cdots+,s)$ appearing in the array $\ArA(a,d)=(a_{i})_{i\in\Z^{n}}$ and is denoted by $\AS(a,d,s)$, that is,
$$
\AS(a,d,s) = \left\{ a + i_1d_1 + \cdots + i_nd_n\ \middle|\ i\in\N^n\ \text{such that}\ i_1+\cdots+i_n\le s-1\right\}.
$$
For $n=1$, the arithmetic progression $\AS(a,d,s)$ is also denoted by $\AP(a,d,s)$.
\end{definition}

\begin{remark}
The following multiset identities hold for arithmetic simplices:
$$
\AS(a,(d_1,\ldots,d_n),s)
\begin{array}[t]{l}
 = \displaystyle\AS(a+(s-1)d_1,(-d_1,d_2-d_1,\ldots,d_n-d_1),s) \\
 = \displaystyle\AS(a+(s-1)d_2,(d_1-d_2,-d_2,d_3-d_2,\ldots,d_n-d_2),s) \\
 \cdots\cdots \\
 = \displaystyle\AS(a+(s-1)d_n,(d_1-d_n,\ldots,d_{n-1}-d_n,-d_n),s),
\end{array}
$$
and $\AS(a,(d_1,\ldots,d_n),s) = \AS\left(a,\left(d_{\pi(1)},\ldots,d_{\pi(n)}\right),s\right)$, where $\pi$ is a permutation of $[1,n]$.
\end{remark}

\begin{figure}
\begin{center}
\begin{tikzpicture}[scale=0.5]

\draw (1,1) -- (1,6);
\draw (2,1) -- (2,6);
\draw (3,2) -- (3,6);
\draw (4,3) -- (4,6);
\draw (5,4) -- (5,6);
\draw (6,5) -- (6,6);
\draw (1,6) -- (6,6);
\draw (1,5) -- (6,5);
\draw (1,4) -- (5,4);
\draw (1,3) -- (4,3);
\draw (1,2) -- (3,2);
\draw (1,1) -- (2,1);

\node at (1.5,5.5) {$0$};
\node at (2.5,5.5) {$1$};
\node at (3.5,5.5) {$2$};
\node at (4.5,5.5) {$3$};
\node at (5.5,5.5) {$4$};
\node at (1.5,4.5) {$2$};
\node at (2.5,4.5) {$3$};
\node at (3.5,4.5) {$4$};
\node at (4.5,4.5) {$0$};
\node at (1.5,3.5) {$4$};
\node at (2.5,3.5) {$0$};
\node at (3.5,3.5) {$1$};
\node at (1.5,2.5) {$1$};
\node at (2.5,2.5) {$2$};
\node at (1.5,1.5) {$3$};

\draw (7,2) -- (7,6);
\draw (8,2) -- (8,6);
\draw (9,3) -- (9,6);
\draw (10,4) -- (10,6);
\draw (11,5) -- (11,6);
\draw (7,6) -- (11,6);
\draw (7,5) -- (11,5);
\draw (7,4) -- (10,4);
\draw (7,3) -- (9,3);
\draw (7,2) -- (8,2);

\node at (7.5,5.5) {$3$};
\node at (8.5,5.5) {$4$};
\node at (9.5,5.5) {$0$};
\node at (10.5,5.5) {$1$};
\node at (7.5,4.5) {$0$};
\node at (8.5,4.5) {$1$};
\node at (9.5,4.5) {$2$};
\node at (7.5,3.5) {$2$};
\node at (8.5,3.5) {$3$};
\node at (7.5,2.5) {$4$};

\draw (12,3) -- (12,6);
\draw (13,3) -- (13,6);
\draw (14,4) -- (14,6);
\draw (15,5) -- (15,6);
\draw (12,6) -- (15,6);
\draw (12,5) -- (15,5);
\draw (12,4) -- (14,4);
\draw (12,3) -- (13,3);

\node at (12.5,5.5) {$1$};
\node at (13.5,5.5) {$2$};
\node at (14.5,5.5) {$3$};
\node at (12.5,4.5) {$3$};
\node at (13.5,4.5) {$4$};
\node at (12.5,3.5) {$0$};

\draw (16,4) -- (16,6);
\draw (17,4) -- (17,6);
\draw (18,5) -- (18,6);
\draw (16,6) -- (18,6);
\draw (16,5) -- (18,5);
\draw (16,4) -- (17,4);

\node at (16.5,5.5) {$4$};
\node at (17.5,5.5) {$0$};
\node at (16.5,4.5) {$1$};

\draw (19,6) -- (20,6) -- (20,5) -- (19,5) -- (19,6);

\node at (19.5,5.5) {$2$};

\end{tikzpicture}
\end{center}
\caption{The arithmetic tetrahedron $\AS(0,(1,2,3),5)$ in $\ZZ{5}$}\label{fig4}
\end{figure}
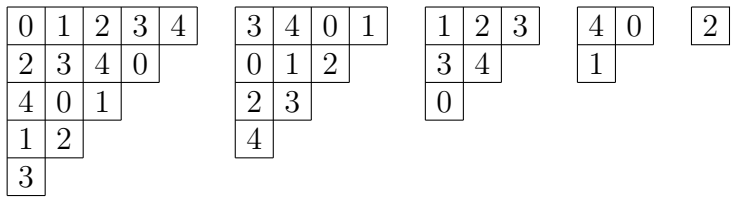

For example, the arithmetic tetrahedron $\AS(0,(1,2,3),5)$ of $\ZZ{5}$ is depicted in Figure~\ref{fig4}. The successive rows of this tetrahedron are the arithmetic triangles $\AS(0,(1,2),5)$, $\AS(3,(1,2),4)$, $\AS(1,(1,2),3)$, $\AS(4,(1,2),2)$ and $\AS(2,(1,2),1)$ of $\ZZ{5}$.

\par Return now to the general case of an ACA of dimension $n-1$, with a weight array $W=(w_j)_{j\in[-r,r]^{n-1}}$ of radius $r\in\N$. Let us denote
$$
\sigma := \sum_{j\in[-r,r]^{n-1}}w_{j}\quad\text{and}\quad \sigma_{k} := \sum_{j\in[-r,r]^{n-1}}j_{k}w_{j},\ \text{for all}\ k\in[1,n-1].
$$
For any integers $a$ and $b$, we let $\gcd(a,b)$ and $\lcm(a,b)$ denote the greatest common divisor and the least common multiplicator of $a$ and $b$, respectively. Let $x\in\ZZ{m}$. We also let $\gcd(x,m)$ denote the greatest common divisor of $m$ and any representant of the residue class $x$.

Using properties of arithmetic simplices, the following theorem, which is the main result of this paper, will be proved.

\begin{theorem}\label{thm1}
Let $n\ge2$ and $m$ be two positive integers such that $\gcd(m,n!)=1$. Suppose that $\sigma$ is invertible modulo $m$. Let $a\in\ZZ{m}$, $d=(d_1,\ldots,d_{n-1})\in(\ZZ{m})^{n-1}$ and $\varepsilon\in\{-1,1\}^{n}$ such that $d_i$, for all $i\in[1,n]$, and $\varepsilon_jd_j-\varepsilon_id_i$, for all distinct integers $i,j\in[1,n]$, are invertible, where $d_n:=\sigma^{-1}\sum_{k=1}^{n-1}\sigma_kd_k$. Then, in the orbit $\orb(\ArA(a,d))$, every $n$-simplex with orientation $\varepsilon$ and of size $s$ is balanced, for all $s\equiv -t \bmod{\lcm(\ord_{m}(\sigma),m)}$, where $t\in[0,n-1]$.
\end{theorem}

\begin{remark}
For any integer $\sigma$ which is invertible modulo $m$, the identity $\lcm(\ord_m(\sigma),m)=\ord_m(\sigma^m)m$ holds. A complete study of this arithmetic function can be found in \cite{Chappelon2010}.
\end{remark}

For $n=2$, $m$ odd and $W=(1,1,0)$, the weight sequence of $\PCA{1}$, we retrieve Theorem~\ref{chap1}. Indeed, in this case, we have $\sigma=2$, $\sigma_1=-1$, $d_1=d$, $d_2=\sigma^{-1}\sigma_1d_1=-2^{-1}d$ and $\varepsilon_1d_1-\varepsilon_2d_2 = \pm 2^{-1}d$ for $\varepsilon=(\pm 1,\mp 1)$, which are invertibles of $\ZZ{m}$.

In the special case of $\PCA{n-1}$, Theorem~\ref{thm1} gives a positive answer to the equivalent problem of the weak Molluzzo problem, in higher dimensions, for an infinite number of values $m$.

\begin{corollary}
Let $n\ge 2$ be a positive integer. For every positive integer $m$ such that $\gcd(m,(3(n-1))!)=1$, there exist infinitely many balanced $n$-simplices of $\ZZ{m}$ generated by $\PCA{n-1}$, for all possible orientations $\varepsilon\in\{-1,1\}^ n$. In the special case of the two orientations $\varepsilon=+\cdots+-$ or $\varepsilon=-\cdots-+$, the existence of an infinite number of such balanced simplices is verified for every $m$ such that $\gcd(m,n!)=1$ for $n$ even and for every $m$ such that $\gcd\left(m,\left(\frac{3n+1}{2}\right)!\right)=1$ for $n$ odd.
\end{corollary}

This paper is organized as follows. After giving some basic results on balanced simplices and orbits of arithmetic arrays generated by ACA in Section~2, we study, in Section~3, arithmetic simplices and we give some sufficient conditions on them to be balanced, for any dimension $n\ge2$. Moreover, in dimension 2 and 3, we also provide necessary conditions on arithmetic triangles and arithmetic tetrahedra for being balanced. This leads to the proof of Theorem~\ref{thm1} in Section~4. Moreover, using the specificities on balanced arithmetic tetrahedra in dimension 2, highlighted in Section~3, we complete Theorem~\ref{thm1} for balanced tetrahedra. In Section~5, we consider the special case where simplices have the additional geometric property of being composed of antisymmetric sequences. This permits us to obtain more results for ACA of dimension 1 generating balanced triangles. Finally, the problem of determining the existence of balanced triangles and tetrahedra generated by $\PCA{1}$ and $\PCA{2}$, for the remaining open cases, is posed in the last section.

\section{Preliminaries}

We begin this section with the terminology on simplices that we will use in the sequel.

\begin{definition}[Vertices, edges, facets and rows of simplices]\label{defverfac}
Let $A=(a_i)_{i\in\Z^n}$ be an infinite array of dimension $n$ of $\ZZ{m}$. Let $\triangle=\triangle(j,\varepsilon,s)$ be the $n$-simplex of size $s$ of $\ZZ{m}$ with principal vertex at position $j\in\Z^n$ in $A$ and with orientation $\varepsilon\in\{-1,1\}^n$. Let $(e_1,e_2,\ldots,e_n)$ denote the canonical basis of the vector space $\Z^n$ and let $e_0:=0_{\Z^n}$.
The $n+1$ \textit{vertices} $\V_0,\ldots,\V_{n}$ of $\triangle$ are defined by $\V_k(\triangle) := a_{j + (s-1)\varepsilon\cdot e_k}$ for all $k\in[0,n]$, where $\V_0(\triangle) := a_{j}$ (principal vertex) and
$$
\V_k(\triangle) := a_{j + \varepsilon\cdot (s-1)e_k} = a_{j_1,\ldots,j_{k-1},j_{k}+\varepsilon_k(s-1),j_{k+1},\ldots,j_n},
$$
for all $k\in[1,n]$. The $\binom{n+1}{2}$ \textit{edges} $\E_{k,l}$ of $\triangle$ are sequences of length $s$ defined by
$$
\E_{k,l}(\triangle) 
\begin{array}[t]{l}
:= \left\{ a_{j + \varepsilon\cdot ((s-1-x)e_k + xe_l)}\ \middle|\ x\in[0,s-1] \right\} \\[2ex]
\ = \left\{ a_{j + \varepsilon\cdot (s-1)e_k)} , a_{j + \varepsilon\cdot ((s-2)e_k + e_l)} , a_{j + \varepsilon\cdot ((s-3)e_k + 2e_l)} , \ldots , a_{j + \varepsilon\cdot (s-1)e_l} \right\},
\end{array}
$$
for all distinct integers $k,l\in[0,n]$. The $n+1$ \textit{facets} $\F_0,\ldots,\F_n$ of $\triangle$ are the $(n-1)$-simplices of size $s$ defined by
$$
\F_k(\triangle) := \displaystyle\left\{ a_{j+\varepsilon\cdot l}\ \middle|\ l\in\N^n \ \text{such that}\ l_k=0\ \text{and}\ l_1+\cdots+l_n \le s-1 \right\},
$$
for all $l\in[1,n]$ and
$$
\F_0(\triangle) := \left\{ a_{j+\varepsilon\cdot l}\ \middle|\ l\in\N^n\ \text{such that}\ l_1+\cdots+l_n = s-1 \right\}.
$$
For every $k\in[0,s-1]$, the $k$th \textit{row} of $\triangle$ is the $(n-1)$-simplex of size $s-k$ defined by
$$
\R_k(\triangle) := \left\{ a_{j+\varepsilon\cdot l} \ \middle|\ l\in\N^n \ \text{such that}\ l_n=k\ \text{and}\ l_1+\cdots+l_{n-1} \le s-k-1\right\}.
$$
\end{definition}

\subsection{Sizes of balanced simplices}

In this subsection, the admissible sizes of balanced simplices are studied. First, the cardinality of an $n$-simplex of size $s$ is determined.

\begin{proposition}\label{propsize}
Let $\triangle$ be an $n$-simplex of size $s$ appearing in an $n$-dimensional array. Then, the multiset cardinality of $\triangle$ is $\left|\triangle\right| = \binom{s+n-1}{n}$.
\end{proposition}

\begin{proof}
By induction on $n$. For $n=1$, $\triangle$ is a finite sequence of length $s$. Thus, $|\triangle|=s=\binom{s}{1}$. Suppose that every $(n-1)$-simplex of size $k$ has cardinality of $\binom{k+n-2}{n-1}$, for all integers $k\ge1$. Now, let $\triangle$ be an $n$-simplex of size $s$. Since, for all $k\in[0,s-1]$, the $k$th row $\R_k$ of $\triangle$ is an $(n-1)$-simplex of size $s-k$, it follows that
$$
|\triangle| = \sum_{k=0}^{s-1}|\R_k| = \sum_{k=0}^{s-1}\binom{s-k+n-2}{n-1} = \sum_{k=n-1}^{s+n-2}\binom{k}{n-1} = \binom{s+n-1}{n}.
$$
This completes the proof.
\end{proof}

The divisibility of $\binom{s+n-1}{n}$ by $m$ is obviously a necessary condition for having a balanced $n$-simplex of $\ZZ{m}$ of size $s$. When $m$ is a composite number, to give all the sizes $s$ for which the binomial $\binom{s+n-1}{n}$ is divisible by $m$ is tedious and not really important here because the results that we obtain in this paper are only for some of them, not for all the admissible sizes. Nevertheless, we can see that the sizes involved in Theorem~\ref{thm1} are admissible for this problem.

\begin{proposition}
Let $n,p,k,s$ be positive integers such that $p$ is prime and $p>n\ge2$. Then, the binomial coefficient $\binom{s+n-1}{n}$ is divisible by $p^k$ if and only if $s\equiv -t\bmod{p^k}$ for $t\in[0,n-1]$.
\end{proposition}

\begin{proof}
Since $p^k$ and $n!$ are relatively prime, we have
$$
\binom{s+n-1}{n} \equiv 0 \pmod{p^k}
\begin{array}[t]{cl}
 \Longleftrightarrow & \displaystyle(s+n-1)(s+n-2)\cdots s \equiv 0 \pmod{p^k}.
\end{array}
$$
Moreover, $n<p$ implies that $p$ can divide at most one factor of $(s+n-1)(s+n-2)\cdots s$. Therefore, $\binom{s+n-1}{n}$ is divisible by $p^k$ if and only if $s+t\equiv0\pmod{p^k}$ for some $t\in[0,n-1]$. This concludes the proof.
\end{proof}

\begin{proposition}
Let $m$ and $n$ be two positive integers such that $\gcd(m,n!)=1$ and let $s$ be a positive integer such that $s\equiv -t\bmod{m}$, where $t\in[0,n-1]$. Then, the binomial coefficient $\binom{s+n-1}{n}$ is divisible by $m$.
\end{proposition}

\begin{proof}
If $s\equiv -t\bmod{m}$, then $(s+n-1)(s+n-2)\cdots(s+1)s$ is divisible by $m$. Finally, since $\gcd(m,n!)=1$, we deduce from Gauss Lemma that $m$ divides $\frac{(s+n-1)(s+n-2)\cdots(s+1)s}{n!}=\binom{s+n-1}{n}$. This concludes the proof.
\end{proof}

\subsection{Projection Theorem}

The result presented in this subsection is useful for proving, by induction on $m$, that multisets of $\ZZ{m}$ are balanced.

For any divisor $\alpha$ of the positive integer $m$, let $\pi_{\alpha}$ denote the canonical projective map $\pi_{\alpha} : \ZZ{m} \longrightarrow \ZZ{\alpha}$. For any multiset $M$ of $\ZZ{m}$, its projection $\pi_{\alpha}(M)$ into $\ZZ{\alpha}$ is the multiset of $\ZZ{\alpha}$ defined by $\m_{\pi_{\alpha}(M)}(y)=\sum_{x\in\pi_{\alpha}^{-1}(\{y\})}\m_M(x)$ for all $y\in\ZZ{\alpha}$.

\begin{theorem}\label{thmproj}
Let $m$ and $\alpha$ be two positive integers such that $m$ is divisible by $\alpha$ and let $M$ be a finite multiset of elements in $\ZZ{m}$. Then, the multiset $M$ is balanced in $\ZZ{m}$ if and only if its projection $\pi_{\alpha}(M)$ is balanced in $\ZZ{\alpha}$ and $\m_M(x+\alpha)=\m_M(x)$ for all $x\in\ZZ{m}$.
\end{theorem}

\begin{proof}
Since $\pi_{\alpha}^{-1}(\{\pi_{\alpha}(x)\})=\left\{x+\lambda \alpha\ \middle|\ \lambda\in\left[0,\frac{m}{\alpha}-1\right]\right\}$ for all $x\in\ZZ{m}$, we obtain that
$
\m_{\pi_{\alpha}(M)}(\pi_{\alpha}(x)) = \sum_{\lambda=0}^{\frac{m}{\alpha}-1}\m_{M}(x+\lambda \alpha)
$
and the result follows.
\end{proof}

\subsection{Orbits of arithmetic arrays}

In this subsection, the orbits of arithmetic arrays are studied in detail. Let $n\ge2$ be a positive integer. First, we show that the arithmetic structure is preserved under the action of $\partial$ for any weight array $W=(w_i)_{i\in[-r,r]^{n-1}}$, of radius $r\in\N$.

\begin{proposition}\label{derivarith}
Let $a\in\ZZ{m}$ and let $d=(d_1,\ldots,d_{n-1})\in{\left(\ZZ{m}\right)}^{n-1}$. Then,
$$
\partial\ArA(a,d) = \ArA\left(\sigma a + \sum_{k=1}^{n-1}\sigma_kd_k, \sigma d\right),
$$
where $\sigma$ and $\sigma_k$ are the coefficients
$$
\sigma := \sum_{j\in[-r,r]^{n-1}}w_{j},\quad \sigma_{k} := \sum_{j\in[-r,r]^{n-1}}j_{k}w_{j},\ \text{for all}\ k\in[1,n-1].
$$
\end{proposition}

\begin{proof}
Let $\ArA(a,d)=\left(a_{i}\right)_{i\in\Z^{n-1}}$ and $\partial\ArA(a,d)=\left(b_{i}\right)_{i\in\Z^{n-1}}$. By definition of $\partial$, for every $i\in\Z^{n-1}$, we have
$$
b_{i} = \hspace{-14pt}\sum_{j\in[-r,r]^{n-1}}\hspace{-14pt} w_{j}a_{i+j} = \hspace{-14pt}\sum_{j\in[-r,r]^{n-1}}\hspace{-14pt} w_{j}\left(a+\sum_{k=1}^{n-1}(i_k+j_k)d_k\right) = \left(\sigma a + \sum_{k=1}^{n-1}\sigma_k d_k\right) + \sum_{k=1}^{n-1}i_k(\sigma d_k).
$$
The result follows.
\end{proof}

\begin{proposition}\label{cor1}
Let $a\in\ZZ{m}$ and let $d=(d_1,\ldots,d_{n-1})\in{\left(\ZZ{m}\right)}^{n-1}$. Then,
$$
\partial^{i}\ArA(a,d) = \ArA\left(\sigma^{i}a+i\sigma^{i-1}\sum_{k=1}^{n-1}\sigma_k d_k,\sigma^{i} d\right),
$$
for all $i\in\N$.
\end{proposition}

\begin{proof}
By induction on $i\in\N$. For $i=0$, we retrieve that $\partial^0\ArA(a,d)=\ArA(a,d)$. For $i\geq 1$, by the recursive definition of $\partial^i$ and Proposition~\ref{derivarith}, we obtain that
$$
\begin{array}[t]{l}
\partial^i\ArA(a,d) = \partial\left(\partial^{i-1}\ArA(a,d)\right) \\
 \ \\
 = \displaystyle\partial\left( \ArA\left(\sigma^{i-1}a+(i-1)\sigma^{i-2}\sum_{k=1}^{n-1}\sigma_k d_k,\sigma^{i-1} d\right)  \right) \\
 \ \\
 = \displaystyle\ArA\left( \sigma\left(\sigma^{i-1}a+(i-1)\sigma^{i-2}\sum_{k=1}^{n-1}\sigma_k d_k\right) + \sum_{k=1}^{n-1}\sigma_k \left(\sigma^{i-1}d_k\right), \sigma\left(\sigma^{i-1} d\right)\right) \\
 \ \\
 = \displaystyle\ArA\left( \sigma^{i}a+ i\sigma^{i-1}\sum_{k=1}^{n-1}\sigma_k d_k , \sigma^{i} d\right).
\end{array}
$$
This concludes the proof.
\end{proof}

Thus, the elements of the orbit of an arithmetic array $\ArA(a,d)$ are entirely determined in function of $a$, $d$, $\sigma$ and $\sigma_k$ for all $k\in[1,n-1]$.

\begin{proposition}\label{prop4}
Let $a\in\ZZ{m}$ and let $d\in{\left(\ZZ{m}\right)}^{n-1}$. Let $\orb(\ArA(a,d))=\left(a_{i}\right)_{i\in\Z^{n-1}\times\N}$ be the orbit of the arithmetic array $\ArA(a,d)$. Then,
$$
a_{i}  = \left(\partial^{i_{n}}\ArA(a,d)\right)_{i_1,\ldots,i_{n-1}}  = \displaystyle\sigma^{i_{n}}\left( a + \sum_{k=1}^{n} i_k d_k \right),
$$
for all $i\in\Z^{n-1}\times\N$, where $d_n:=\sigma^{-1}\sum_{k=1}^{n-1}\sigma_k d_k$.
\end{proposition}

\begin{proof}
Directly comes from Proposition~\ref{cor1}.
\end{proof}

\begin{remark}\label{rem1}
For $\orb(\ArA(a,d))=\left(a_{i}\right)_{i\in\Z^{n-1}\times\N}$ and for every $(i_1,\ldots,i_{n-1})\in\Z^{n-1}$, the sequence $\left( a_{i_1,\ldots,i_{n-1},i_n} \right)_{i_n\in\N}$ is the arithmetico-geometric sequence with first element $a+i_1d_1+\cdots+i_{n-1}d_{n-1}$, with common difference $d_n:=\sigma^{-1}\sum_{k=1}^{n-1}\sigma_k d_k$ and common ratio $\sigma$.
\end{remark}

We deduce from Proposition~\ref{prop4} that two distinct ACA can generate the same orbit from an arithmetic array. For instance, for any ACA of weight array $W=(w_i)_{i\in{[-r,r]}^{n-1}}$ of radius $r$, we can consider the ACA of weight array $\overline{W}=(\overline{w_i})_{i\in{[-1,1]}^{n-1}}$ of radius $1$ defined by
$$
\overline{w_i} = \left\lbrace
\begin{array}{ll}
\sigma(W)-\displaystyle\sum_{k=1}^{n-1}\sigma_k(W) & ,\text{ if } i=0_{\Z^{n}}, \\
\sigma_k(W) & , \text{ if } i=e_k,\text{ for all } k\in[1,n-1], \\
0 & , \text{ otherwise}.\\
\end{array}\right.
$$
Then, it is clear that we have
$$
\sigma(\overline{W}) = \sum_{i\in{[-1,1]}^{n-1}}\overline{w_i} = \sigma(W),
$$
and
$$
\sigma_k(\overline{W}) = \sum_{i\in{[-1,1]}^{n-1}}i_k\overline{w_i} = \sigma_k(W),
$$
for all $k\in[1,n-1]$. Therefore, in the sequel of this paper, the coefficients $\sigma$ and $\sigma_k$ will be more important than the elements of the weight array $W$ themselves.

Now, we prove that, in the orbit of an arithmetic array of $\ZZ{m}$, if there exists a balanced simplex of sufficiently large size, then $\sigma$ is invertible modulo $m$.

\begin{proposition}
Let $a\in\ZZ{m}$ and let $d\in{\left(\ZZ{m}\right)}^{n-1}$. In the orbit $\orb(\ArA(a,d))$, if an $n$-simplex of size $s> \frac{n}{m-1}+\frac{3}{2}+\sqrt{\frac{mn^2}{(m-1)^2}+\frac{1}{4}}$ is balanced in $\ZZ{m}$, then $\sigma$ is invertible modulo $m$.
\end{proposition}

\begin{proof}
Let $j\in\Z^{n-1}\times\N$ and let $\varepsilon\in\{-1,+1\}^{n}$. Suppose that $\triangle(j,\varepsilon,s)$ is balanced and that $\sigma$ is not invertible modulo $m$. Without loss of generality, we can suppose that $\sigma\equiv 0\bmod{m}$. If not, we consider the projection into $\ZZ{\gcd(\sigma,m)}$ and then we have $\sigma\equiv 0\bmod{\gcd(\sigma,m)}$. By Proposition~\ref{cor1}, we know that, in the case $\sigma\equiv0\bmod{m}$, we have $\partial^i\ArA(a,d)=\ArA(0,0)$, the constant array equal to zero in $\ZZ{m}$, for all $i\ge2$. Therefore, the $k$th row of $\triangle(j,\varepsilon,s)$ is uniquely composed of elements equal to zero, for all $k$ such that $j_n+k\varepsilon_n\ge 2$. Thus $\triangle(j,\varepsilon,s)$ contains at least $\binom{s+n-3}{n}$ elements equal to zero by Proposition~\ref{propsize}. Moreover, since $\triangle(j,\varepsilon,s)$ is balanced, it must contain at least $m\binom{s+n-3}{n}$ elements. It follows that we have
$$
m\binom{s+n-3}{n} \le \binom{s+n-1}{n}
\begin{array}[t]{l}
\displaystyle\Longleftrightarrow\ m\prod_{i=s-2}^{s+n-3}i \le \prod_{i=s}^{s+n-1}i \\ \ \\
\displaystyle\Longleftrightarrow\ m(s-2)(s-1) \le (s+n-2)(s+n-1).
\end{array}
$$
Thus
$$
(m-1)s^2-(2n+3m-3)s-(n^2-3n-2m+2)\le 0
$$
and we deduce that this is possible if only if
$$
\frac{n}{m-1} + \frac{3}{2} - \sqrt{\frac{mn^2}{(m-1)^2}+\frac{1}{4}} \le s \le \frac{n}{m-1} + \frac{3}{2} + \sqrt{\frac{mn^2}{(m-1)^2}+\frac{1}{4}},
$$
in contradiction with the hypothesis that $s>\frac{n}{m-1} + \frac{3}{2} + \sqrt{\frac{mn^2}{(m-1)^2}+\frac{1}{4}}$. This concludes the proof.
\end{proof}

\begin{remark}
For all integers $m\ge 2$, we have $\frac{n}{m-1} + \frac{3}{2} + \sqrt{\frac{mn^2}{(m-1)^2}+\frac{1}{4}}  < \frac{5n+3}{2}$.
\end{remark}

This is the reason why we suppose, in the sequel of this paper, that $\sigma$ is invertible modulo $m$. We end this section by showing that a simplex in the orbit of an arithmetic array can be decomposed into arithmetic subsimplices.

\begin{proposition}\label{propsub}
Let $a\in\ZZ{m}$ and let $d=(d_1,\ldots,d_{n-1})\in\left(\ZZ{m}\right)^{n-1}$. Let $\alpha$ and $s$ be two positive integers such that $\alpha$ is divisible by $\ord_m(\sigma)$ and $s\equiv -t \bmod{\alpha}$, where $t\in[0,n-1]$, and let $\varepsilon\in\{-1,1\}^{n}$. Let $\triangle(j,\varepsilon,s)$ be the $n$-simplex appearing in the orbit $\orb(\ArA(a,d))=(a_i)_{i\in\Z^{n-1}\times\N}$. Then, for every $k\in[0,\alpha-1]^{n}$, the subsimplex
$$
\SuS_k := \left\{ a_{j+\varepsilon\cdot(k+\alpha l)}\ \middle|\ l\in\N^{n}\ \text{such that}\ (k_1+\alpha l_1) + \cdots + (k_n+\alpha l_n) \le s-1 \right\},
$$
obtained from $\triangle(j,\varepsilon,s)$ by extracting one term every $\alpha$ in each component, is the arithmetic simplex
$$
\SuS_k = \AS\left(a_{j+\varepsilon\cdot k},\alpha\sigma^{j_n+\varepsilon_n k_n}\varepsilon\cdot\tilde{d},\left\lceil\frac{s}{\alpha}\right\rceil-\left\lfloor\frac{\sum_{u=1}^{n}k_u+t}{\alpha}\right\rfloor\right),
$$
where $\tilde{d}=\left(d_1,\ldots,d_{n-1},\sigma^{-1}\sum_{u=1}^{n-1}\sigma_ud_u\right)$.
\end{proposition}

\begin{proof}
Let $k\in[0,\alpha-1]^{n}$. As already observed in Remark~\ref{rem1}, for every $(i_1,\ldots,i_{n-1})\in\Z^{n-1}$, the sequence $\left( a_{i_1,\ldots,i_{n-1},i_n} \right)_{i_n\in\N}$ is the arithmetico-geometric sequence whose first element is $a+i_1d_1+\cdots+i_{n-1}d_{n-1}$, with common difference $\sigma^{-1}\sum_{u=1}^{n-1}\sigma_u d_u$ and with common ratio $\sigma$. Since $\alpha$ is a multiple of $\ord_m(\sigma)$, it follows that, for every $(l_1,\ldots,l_{n-1})\in\N^{n-1}$, the sequence $\left(a_{j+\varepsilon\cdot(k+\alpha l)}\right)_{l_n\in\N}$ is arithmetic, with common difference $\alpha\sigma^{j_n+\varepsilon_n k_n-1}\varepsilon_n\sum_{u=1}^{n-1}\sigma_u d_u$. Indeed, from Proposition~\ref{prop4}, we obtain
$$
a_{j+\varepsilon\cdot(k+\alpha l)} = \sigma^{j_n+\varepsilon_n(k_n+\alpha l_n)}\left( a + \sum_{u=1}^{n}(j_u+\varepsilon_u(k_u+\alpha l_u))d_u \right),
$$
where $d_n:=\sigma^{-1}\sum_{u=1}^{n-1}\sigma_ud_u$. Since $\alpha$ is a multiple of $\ord_n(\sigma)$, we have
$$
\sigma^{j_n+\varepsilon_n(k_n+\alpha l_n)}=\sigma^{j_n+\varepsilon_n k_n}.
$$
Thus
$$
a_{j+\varepsilon\cdot (k+\alpha l)} = \sigma^{j_n+\varepsilon_n k_n} \left[ \left( a + \sum_{u=1}^{n}(j_u+\varepsilon_u k_u)d_u \right) + \alpha\sum_{u=1}^{n}l_u \varepsilon_u d_u \right].
$$
By Proposition~\ref{prop4} again, for $a_{j+\varepsilon\cdot k}$,
$$
a_{j+\varepsilon\cdot (k+\alpha l)} = a_{j+\varepsilon\cdot k} + \alpha\sigma^{j_n+\varepsilon_n k_n}\sum_{u=1}^{n}l_u(\varepsilon_u d_u),
$$
for every $l\in\N^{n}$. Therefore the subsimplex $\SuS_k$ is an arithmetic simplex whose principal vertex is $a_{j+\varepsilon\cdot k}$ and with common difference $\alpha\sigma^{j_n+\varepsilon_n k_n}\varepsilon\cdot\tilde{d}$, where $\tilde{d}=(d_1,d_2,\ldots,d_n)$. It remains to determine the size of this arithmetic simplex. Let $\lambda:=\left\lceil\frac{s}{\alpha}\right\rceil$ and $\mu:=\left\lfloor\frac{\sum_{u=1}^{n}k_u+t}{\alpha}\right\rfloor$. In other words, we have $s=\lambda\alpha-t$ and
$$
\mu\alpha-t\le \sum_{u=1}^{n}k_u \le (\mu+1)\alpha-t-1.
$$
For any $l\in\N^{n}$, the inequality
$$
\sum_{u=1}^{n}(k_u+\alpha l_u) \le s-1 = \lambda\alpha-t-1
$$
is then equivalent to
$$
\sum_{u=1}^{n}l_u \le \lambda-\mu-1.
$$
Therefore the arithmetic simplex $\SuS_k$ is of size $\lambda-\mu$. This completes the proof.
\end{proof}

From the previous proposition, we know that every $n$-simplex $\triangle$ of size $\lambda\alpha-t$, where $\alpha$ is a multiple of $\ord_m(\sigma)$ and $t\in[0,n-1]$, appearing in the orbit of an arithmetic array can be decomposed into $\alpha^n$ arithmetic $n$-simplices of sizes in $[\lambda-(n-1),\lambda]$. Therefore, in next section, the arithmetic simplices will be studied in detail.

\section{Balanced arithmetic simplices}

In this section, we will see that arithmetic simplices are a source of balanced multisets of $\ZZ{m}$. First, we show, in the general case $n\ge 1$, that there exists sufficient conditions on arithmetic simplices for being balanced. After that, in dimension $n=2$ and $n=3$, i.e., for arithmetic triangles and arithmetic tetrahedra, necessary conditions for being balanced are also given.

\subsection{The general case: in dimension $\mathbf{n\ge1}$}

We begin this subsection by showing that, when $n\ge2$, the edges, the facets and the rows of an arithmetic simplex are also arithmetic.

\begin{proposition}\label{facearith}
Let $a\in\ZZ{m}$, $d=(d_1,\ldots,d_n)\in(\ZZ{m})^{n}$ and let $s$ be a positive integer. Let $\triangle:=\AS(a,d,s)$ and let $d_0:=0$. Then, we have
$$
\V_i(\triangle) = a+(s-1)d_{i},
$$
for all $i\in[0,n]$,
$$
\E_{i,j}(\triangle) = \AP(\V_i(\triangle),d_j-d_i,s),
$$
for all distinct integers $i,j\in[0,n]$,
$$
\R_{i}(\triangle) = \AS(a+i d_n,(d_1,\ldots,d_{n-1}),s-i),
$$
for all $i\in[0,n]$,
$$
\F_i(\triangle) = \AS(a,(d_1,\ldots,d_{i-1},d_{i+1},\ldots,d_n),s),
$$
for all $i\in[1,n]$, and
$$
\F_0(\triangle)= \AS(a+(s-1)d_1,(d_2-d_1,\ldots,d_n-d_1),s).
$$
Moreover, for all $i\in[0,n]$, we have
$$
\triangle\setminus \F_i(\triangle) = \AS(a+d_i,d,s-1).
$$
\end{proposition}

\begin{proof}
By Definition~\ref{defarith} and Definition~\ref{defverfac}.
\end{proof}

The following theorem, which gives sufficient conditions on arithmetic simplices for being balanced, is the main result of this section.

\begin{theorem}\label{thm2}
Let $n$ and $m$ be two positive integers such that $\gcd(m,n!)=1$. Let $a\in\ZZ{m}$ and let $d=(d_1,\ldots,d_n)\in{\left(\ZZ{m}\right)}^{n}$ such that $d_i$, for all $i\in[1,n]$, and $d_j-d_i$, for all distinct integers $i,j\in[1,n]$, are invertible. Then, the arithmetic simplex $\AS(a,d,s)$ is balanced for all $s\equiv -t\bmod{m}$, with $t\in[0,n-1]$.
\end{theorem}

The proof of this theorem is based on the following result, which is a key lemma in this paper.

\begin{lemma}\label{keylemma}
Let $n\ge 2$, $m$ and $s$ be positive integers. Let $a\in\ZZ{m}$ and let $d\in(\ZZ{m})^{n}$. Then, the multiplicity function of $\triangle=\AS(a,d,s)$ verifies
$$
\m_{\triangle}(x+d_j) - \m_{\triangle}(x+d_i) = \m_{\F_j(\triangle)}(x+d_j) - \m_{\F_i(\triangle)}(x+d_i),
$$
for all $x\in\ZZ{m}$ and for all distinct integers $i,j\in[0,n]$, where $d_0:=0$.
\end{lemma}

\begin{proof}
Let $i$ and $j$ be distinct integers in $[0,n]$. Since, from Proposition~\ref{facearith}, we have
$$
\triangle\setminus \F_k(\triangle) = \AS(a+d_k,d,s-1),
$$
for all integers $k\in[0,n]$, it follows that the arithmetic simplex $\triangle\setminus \F_j(\triangle)$ is a translate of $\triangle\setminus \F_i(\triangle)$ by $d_j-d_i$. Thus, for all $x\in\ZZ{m}$, we have
$$
\m_{\triangle\setminus \F_j(\triangle)}(x+d_j) = \m_{\triangle\setminus \F_i(\triangle)}(x+d_i).
$$
Therefore,
$$
\m_{\triangle}(x+d_j) - \m_{\F_j(\triangle)}(x+d_j) = \m_{\triangle}(x+d_i) - \m_{\F_i(\triangle)}(x+d_i)
$$
for all $x\in\ZZ{m}$.
\end{proof}

\begin{proof}[Proof of Theorem~\ref{thm2}]
By induction on $n$.

For $n=1$ and $s\equiv 0 \bmod{m}$, the arithmetic simplex $\AS(a,d,s)$ is simply an arithmetic progression with invertible common difference $d\in\ZZ{m}$ and of length a multiple of $m$. Therefore $\AS(a,d,s)$ is balanced in this case.

Suppose now that $n\geq2$ and that the result is true in dimension $\le n-1$. We distinguish different cases depending on the residue class of $s$ modulo $m$.

\begin{case}
For $s\equiv -t \bmod{m}$, with $t\in[0,n-2]$.

Let $\triangle=\AS(a,d,s)$. First, from Proposition~\ref{facearith}, we know that
$$
\F_1(\triangle) = \AS(a,(d_2,\ldots,d_n),s)\ \text{and}\ \F_2(\triangle) = \AS(a,(d_1,d_3,\ldots,d_n),s).
$$
Moreover, since $d_i$, for all $i\in[1,n]$, and $d_j-d_i$, for all distinct integers $i,j\in[1,n]$, are invertible, we obtain by the induction hypothesis that the facets $\F_1(\triangle)$ and $\F_2(\triangle)$ are balanced simplices of dimension $n-1$. Therefore, their multiplicity functions $\m_{\F_1(\triangle)}$ and $\m_{\F_2(\triangle)}$ are constant on $\ZZ{m}$, equal to
$$
\m_{\F_1(\triangle)}(x) = \m_{\F_2(\triangle)}(x) = \frac{1}{m}\binom{s+n-2}{n-1},
$$
for all $x\in\ZZ{m}$. By Lemma~\ref{keylemma}, we obtain
$$
\m_{\triangle}(x+d_1)-\m_{\triangle}(x+d_2) = \m_{\F_1(\triangle)}(x+d_1) - \m_{\F_2(\triangle)}(x+d_2) = 0,
$$
for all $x\in\ZZ{m}$. Finally, since $d_2-d_1$ is invertible and
$$
\m_{\triangle}(x+(d_1-d_2)) = \m_{\triangle}(x)
$$
for all $x\in\ZZ{m}$, we conclude that $\m_\triangle$ is a constant function and so the arithmetic simplex $\triangle$ is balanced in $\ZZ{m}$.
\end{case}

\begin{case}
For $s\equiv -(n-1)\bmod{m}$.

The arithmetic simplex $\AS(a,(d_1,\ldots,d_n),s)$ can be seen as the multiset difference $\triangle\setminus \F_1(\triangle)$, where
$$
\triangle=\AS(a-d_1,(d_1,\ldots,d_n),s+1).
$$
Since $s+1\equiv -(n-2)\bmod{m}$, it follows, from Case~1 and from the induction hypothesis, respectively, that $\triangle=\AS(a-d_1,(d_1,\ldots,d_n),s+1)$ and $\F_1(\triangle)=\AS(a-d_1,(d_2,\ldots,d_n),s+1)$ are balanced. The multiset difference of balanced multisets is also balanced. This concludes the proof.
\end{case}
\vspace{-19pt}
\end{proof}

\subsection{In dimension 2}

In this subsection, we only consider arithmetic triangles over $\ZZ{m}$. Necessary conditions on the common differences $d_1$, $d_2$ and $d_2-d_1$ of $\AS(a,(d_1,d_2),s)$, depicted in Figure~\ref{fig10}, for being balanced in $\ZZ{m}$ are determined.

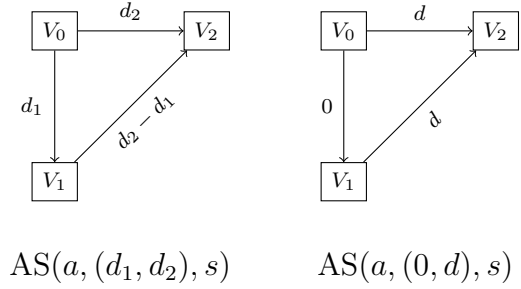
\begin{figure}
\begin{tabular}{c@{\hspace{30pt}}c}
\begin{tikzpicture}[scale=0.5]
\node[draw] at (0,0) (0) {\tiny $V_0$};
\node[draw] at (0,-4) (1) {\tiny $V_1$};
\node[draw] at (4,0) (2) {\tiny $V_2$};

\path[->] (0) edge node [left] {\tiny $d_1$} (1);
\path[->] (0) edge node [above] {\tiny $d_2$} (2);
\path[->] (1) edge node [sloped,below] {\tiny $d_2-d_1$} (2);
\end{tikzpicture}
&
\begin{tikzpicture}[scale=0.5]
\node[draw] at (0,0) (0) {\tiny $V_0$};
\node[draw] at (0,-4) (1) {\tiny $V_1$};
\node[draw] at (4,0) (2) {\tiny $V_2$};

\path[->] (0) edge node [left] {\tiny $0$} (1);
\path[->] (0) edge node [above] {\tiny $d$} (2);
\path[->] (1) edge node [sloped,below] {\tiny $d$} (2);
\end{tikzpicture} \\ \ \\
$\AS(a,(d_1,d_2),s)$ & $\AS(a,(0,d),s)$ \\
\end{tabular}
\caption{Common differences of arithmetic triangles}\label{fig10}
\end{figure}

\begin{theorem}\label{balasn2}
Let $m$ and $s$ be two positive integers and let $a,d_1,d_2\in\ZZ{m}$. If the arithmetic triangle $\AS(a,(d_1,d_2),s)$ is balanced, then the common differences $d_1$, $d_2$ and $d_2-d_1$ are all invertible.
\end{theorem}

The proof of this theorem is based on the following lemma, where the multiplicity function of the arithmetic triangle $\AS(a,(0,d),s)$, with $d$ invertible, is explicitly given.

\begin{lemma}\label{prop1}
Let $m$ and $s$ be two positive integers and let $a,d\in\ZZ{m}$ such that $d$ is invertible. Then, the arithmetic triangle $\triangle=\AS(a,(0,d),s)$ is not balanced in $\ZZ{m}$. Moreover, if $s=\lambda m+\mu$ is the Euclidean division of $s$ by $m$, we have
$$
\m_{\triangle}(a+id) = \binom{\lambda+1}{2}m + \left\lceil\frac{s-i}{m}\right\rceil(\mu-i),
$$
for all integers $i\in[0,m-1]$.
\end{lemma}

\begin{proof}
Let $s=\lambda m +\mu$ be the Euclidean division of $s$ by $m$. As represented in Figure~\ref{fig10}, the common differences of $\triangle=\AS(a,(0,d),s)$ are $0$, $d$ and $d$. Then, for every integer $j\in[0,s-1]$, the $j$th row $\R_j$ of $\triangle$ is the constant sequence of length $s-j$ equal to $a+jd$, that is, $\R_j(\triangle)=\AP(a+jd,0,s-j)$. Thus, the multiplicity function $\m_{\triangle}$ is determined by
$$
\m_{\triangle}(a+id)
\begin{array}[t]{l}
 = \displaystyle\sum_{j=0}^{s-1}\m_{\R_j(\triangle)}(a+id) = \sum_{j=0}^{\left\lceil\frac{s-i}{m}\right\rceil}\m_{\R_{jm+i}(\triangle)}(a+id) \\
 \ \\
 = \displaystyle\sum_{j=0}^{\left\lceil\frac{s-i}{m}\right\rceil}(s-(jm+i)) = \sum_{j=0}^{\left\lceil\frac{s-i}{m}\right\rceil}((\lambda-j)m+(\mu-i)) \\
 \ \\
 = \displaystyle\binom{\lambda+1}{2}m + \left\lceil\frac{s-i}{m}\right\rceil(\mu-i)
\end{array}
$$
for all integers $i\in[0,m-1]$. It follows that
$$
\m_{\triangle}(a) > \m_{\triangle}(a+d) > \m_{\triangle}(a+2d) > \cdots\cdots > \m_{\triangle}(a+(m-1)d).
$$
Since its multiplicity function is not constant on $\ZZ{m}$, the arithmetic triangle $\triangle$ is not balanced. This completes the proof.
\end{proof}

\begin{proof}[Proof of Theorem~\ref{balasn2}]
Let $\triangle=\AS(a,(d_1,d_2),s)$ be an arithmetic triangle and suppose that there exists at least one common difference that is not invertible. Without loss of generality, suppose that $d_1$ is not invertible. Then, we consider the projection of $\triangle$ into $\ZZ{\alpha}$, where $\alpha=\gcd(d_1,m)\geq 2$. Then,
$$
\pi_\alpha(\triangle)=\AS(\pi_{\alpha}(a),(0,\pi_{\alpha}(d_2)),s).
$$
If $\pi_{\alpha}(d_2)$ is invertible in $\ZZ{\alpha}$, then $\triangle$ is not balanced in $\ZZ{m}$ since $\pi_{\alpha}(\triangle)$ is not balanced in $\ZZ{\alpha}$ by Lemma~\ref{prop1}. Otherwise, if $\pi_{\alpha}(d_2)$ is not invertible in $\ZZ{\alpha}$, the result follows since the projection of $\pi_{\alpha}(\triangle)$ into $\ZZ{\beta}$, where $\beta=\gcd(\pi_{\alpha}(d_1),\alpha)\geq 2$, is the constant triangle uniquely composed of elements $\pi_{\beta}(\pi_{\alpha}(a))$, which is obviously not balanced in $\ZZ{\beta}$.
\end{proof}

It follows from this theorem that there does not exist balanced arithmetic triangles in $\ZZ{m}$ for $m$ even. Nevertheless, in the case where $m$ is an even number, the multiplicity function of an arithmetic triangle of $\ZZ{m}$ can be completely determined when exactly two of the three common differences $d_1,d_2,d_1-d_2$ are invertible and the size $s$ is such that $s\equiv0$ or $-1\bmod{m}$.

\begin{proposition}\label{propbalasn2}
Let $m$ and $s$ be two positive integers such that $s\equiv 0$ or $-1\bmod{m}$. Let $a,d_1,d_2\in\ZZ{m}$ and let $\triangle=\AS(a,(d_1,d_2),s)$. If $d_2$ and $d_2-d_1$ are invertible, then
$$
\m_{\triangle}(x) = \m_{\triangle}(x+\gcd(d_1,m)),
$$
for all $x\in\ZZ{m}$, and
$$
\m_{\triangle}(a+id_2) = \frac{1}{m}\binom{s+1}{2} + \left\lceil\frac{s}{m}\right\rceil\left(\frac{\gcd(d_1,m)-1}{2}-i\right),
$$
for all integers $i\in[0,\gcd(d_1,m)-1]$.
\end{proposition}

\begin{proof}
Since $d_2$ and $d_2-d_1$ are invertible and $s\equiv 0$ or $-1\bmod{m}$, it follows that $\F_1(\triangle)=\AP(a,d_2,s)$ and $\F_0(\triangle)=\AP(a+(s-1)d_1,d_2-d_1,s)$ are balanced in $\ZZ{m}$. Therefore, $\m_{\F_1(\triangle)}(x)=\m_{\F_0(\triangle)}(x)=\frac{1}{m}\binom{s+1}{2}$ for all $x\in\ZZ{m}$. It follows, from Lemma~\ref{keylemma}, that
$$
\m_{\triangle}(x) - \m_{\triangle}(x+d_1) = \m_{\F_0(\triangle)}(x) - \m_{\F_1(\triangle)}(x+d_1) = 0,
$$
for all $x\in\ZZ{m}$. The first identity is then proved.

The second identity, in the case where $d_1$ is invertible, comes from the fact that $\triangle$ is balanced, by Theorem~\ref{thm2}, in this case. Suppose now that $\alpha=\gcd(d_1,m)\ge 2$ and consider the projected triangle $\pi_{\alpha}(\triangle)$ in $\ZZ{\alpha}$. Then, since $\pi_{\alpha}(\triangle)=\AS(\pi_{\alpha}(a),(0,\pi_{\alpha}(d_2)),s)$ and $\pi_{\alpha}(d_2)$ is invertible in $\ZZ{\alpha}$, the multiplicity function $\m_{\pi_{\alpha}(\triangle)}$ is entirely determined by Lemma~\ref{prop1}. Let $s=\lambda m+\mu$ be the Euclidean division of $s$ by $m$, where $\mu\in\{0,m-1\}$. Then, the Euclidean division of $s$ by $\alpha$ is $s=\lambda\frac{m}{\alpha}\alpha$, if $\mu=0$, and $\left((\lambda+1)\frac{m}{\alpha}-1\right)\alpha + (\alpha-1)$, if $\mu=m-1$. The multiplicity function $\m_{\pi_{\alpha}(\triangle)}$ verifies
$$
\m_{\pi_{\alpha}(\triangle)}(\pi_{\alpha}(a)+i\pi_{\alpha}(d)) = 
\left\{\begin{array}{ll}
\displaystyle\binom{\lambda \frac{m}{\alpha}+1}{2}\alpha - \left\lceil\frac{s-i}{\alpha}\right\rceil i  & ,\ \text{if}\ \mu=0, \\
 \ \\
\displaystyle\binom{(\lambda+1) \frac{m}{\alpha}}{2}\alpha + \left\lceil\frac{s-i}{\alpha}\right\rceil (\alpha-1-i)  & ,\ \text{if}\ \mu=m-1, \\
\end{array}\right.
$$
for all $i\in[0,\alpha-1]$. Moreover, by the first identity, since $\m_{\triangle}(x+\alpha)=\m_{\triangle}(x)$ for all $x\in\ZZ{m}$, we have
$$
\m_{\triangle}(a+id) = \frac{\alpha}{m}\times \m_{\pi_{\alpha}(\triangle)}(\pi_{\alpha}(a)+i\pi_{\alpha}(d)),
$$
for all $i\in[0,\alpha-1]$. Finally, since
$$
\left\lceil\frac{s-i}{\alpha}\right\rceil = \left\{\begin{array}{ll}
\displaystyle\lambda\frac{m}{\alpha} & ,\ \text{if}\ \mu=0,\ i\in[0,\alpha-1], \\ \ \\
\displaystyle\left(\lambda+1\right)\frac{m}{\alpha} & ,\ \text{if}\ \mu=m-1,\ i\in[0,\alpha-2], \\ \ \\
\displaystyle\left(\lambda+1\right)\frac{m}{\alpha}-1 & ,\ \text{if}\ \mu=m-1,\ i=\alpha-1, \\
\end{array}\right.
$$
it follows that
$$
\m_{\triangle}(a+id) = \binom{\lambda\frac{m}{\alpha}+1}{2}\frac{\alpha^{2}}{m}-\lambda i = \frac{\lambda(\lambda m+\alpha)}{2} - \lambda i = \frac{1}{m}\binom{\lambda m +1}{2} + \lambda \left(\frac{\alpha-1}{2}-i\right),
$$
for all $i\in[0,\alpha-1]$ if $\mu=0$, and
$$
\m_{\triangle}(a+id)
\begin{array}[t]{l}
 = \displaystyle\binom{(\lambda+1)\frac{m}{\alpha}}{2}\frac{\alpha^2}{m} + (\lambda+1)(\alpha-1-i) \\ \ \\
 = \displaystyle\frac{(\lambda+1)((\lambda+1)m-\alpha)}{2} + (\lambda+1)(\alpha-1-i) \\ \ \\
 = \displaystyle\frac{1}{m}\binom{(\lambda+1)m}{2} + (\lambda+1)\left(\frac{\alpha-1}{2}-i\right),
\end{array}
$$
for all $i\in[0,\alpha-1]$ if $\mu=m-1$. This completes the proof.
\end{proof}

For example, for $m=12$, $a=0$, $d_1=1$ and $d_2=5$, we obtain that $d_2-d_1=4$, $\gcd(d_1-d_2,m)=4$ and the multiplicity function of $\triangle=\AS(a,(d_1,d_2),m)$ is given in Table~\ref{tab1}.

\begin{table}
\begin{tabular}{|c||c|c|c|c|c|c|c|c|c|c|c|c|}
\hline
$x$ & $0$ & $1$ & $2$ & $3$ & $4$ & $5$ & $6$ & $7$ & $8$ & $9$ & $10$ & $11$ \\
\hline
$\m_{\triangle}(x)$ & $5$ & $6$ & $7$ & $8$ & $5$ & $6$ & $7$ & $8$ & $5$ & $6$ & $7$ & $8$ \\
\hline
\end{tabular} \\ \ \\
\caption{Multiplicity function of $\triangle=\AS(0,(1,5),12)$ in $\ZZ{12}$}\label{tab1}
\end{table}

\subsection{In dimension 3}

In this subsection, we only consider the arithmetic tetrahedron $\AS(a,(d_1,d_2,d_3),s)$ in $\ZZ{m}$. We determine necessary and sufficient conditions on the common differences $d_1$, $d_2$, $d_3$, $d_2-d_1$, $d_3-d_2$ and $d_1-d_3$ of $\AS(a,(d_1,d_2,d_3),s)$, depicted in Figure~\ref{fig5}, for being balanced in $\ZZ{m}$.

\begin{figure}
\begin{tikzpicture}[scale=0.5]
\node[draw] at (0,0) (0) {\tiny $V_0$};
\node[draw] at (0,-4) (1) {\tiny $V_1$};
\node[draw] at (4,0) (2) {\tiny $V_2$};
\node[draw] at (4,2) (3) {\tiny $V_3$};

\path[->] (0) edge node [left] {\tiny $d_1$} (1);
\path[->] (0) edge node [above] {\tiny $d_2$} (2);
\path[->] (0) edge node [sloped,above] {\tiny $d_3$} (3);
\path[->] (2) edge node [right] {\tiny $d_3-d_2$} (3);
\path[->] (1) edge node [sloped,below] {\tiny $d_2-d_1$} (2);
\path[->] (3) edge node [sloped,above left] {\tiny $d_1-d_3$} (1);
\end{tikzpicture}
\caption{Common differences of an arithmetic tetrahedron}\label{fig5}
\end{figure}
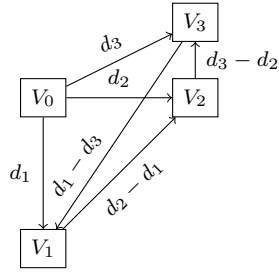

\begin{definition}[Adjacent common differences]
Among the six common differences $d_1$, $d_2$, $d_3$, $d_2-d_1$, $d_3-d_2$ and $d_1-d_3$ of $\AS(a,(d_1,d_2,d_3),s)$, two of them are said to be adjacent if they have a vertex in common. The couple of non-adjacent common differences of $\AS(a,(d_1,d_2,d_3),s)$ are $(d_1,d_3-d_2)$, $(d_2,d_1-d_3)$ and $(d_3,d_2-d_1)$. The twelve other couples of common differences are said to be adjacent (See Figure~\ref{fig3}).
\end{definition}

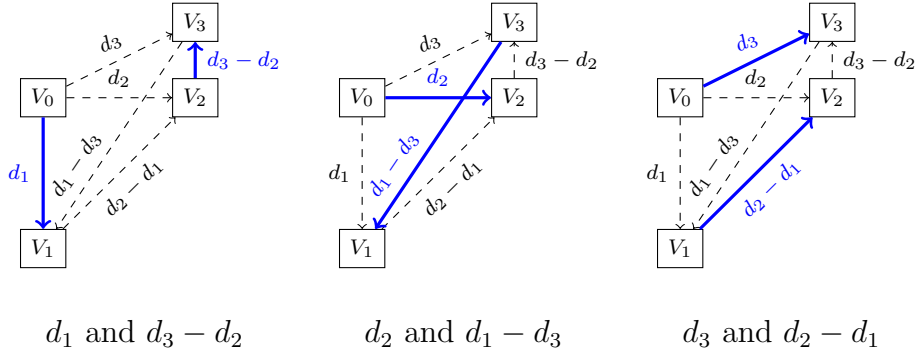
\begin{figure}
\begin{tabular}{ccc}
\begin{tikzpicture}[scale=0.5]
\node[draw] at (0,0) (0) {\tiny $V_0$};
\node[draw] at (0,-4) (1) {\tiny $V_1$};
\node[draw] at (4,0) (2) {\tiny $V_2$};
\node[draw] at (4,2) (3) {\tiny $V_3$};

\path[->,very thick,color=blue] (0) edge node [left] {\tiny $d_1$} (1);
\path[->,dashed] (0) edge node [above] {\tiny $d_2$} (2);
\path[->,dashed] (0) edge node [sloped,above] {\tiny $d_3$} (3);
\path[->,very thick,color=blue] (2) edge node [right] {\tiny $d_3-d_2$} (3);
\path[->,dashed] (1) edge node [sloped,below] {\tiny $d_2-d_1$} (2);
\path[->,dashed] (3) edge node [sloped,above left] {\tiny $d_1-d_3$} (1);
\end{tikzpicture}
&
\begin{tikzpicture}[scale=0.5]
\node[draw] at (0,0) (0) {\tiny $V_0$};
\node[draw] at (0,-4) (1) {\tiny $V_1$};
\node[draw] at (4,0) (2) {\tiny $V_2$};
\node[draw] at (4,2) (3) {\tiny $V_3$};

\path[->,dashed] (0) edge node [left] {\tiny $d_1$} (1);
\path[->,very thick,color=blue] (0) edge node [above] {\tiny $d_2$} (2);
\path[->,dashed] (0) edge node [sloped,above] {\tiny $d_3$} (3);
\path[->,dashed] (2) edge node [right] {\tiny $d_3-d_2$} (3);
\path[->,dashed] (1) edge node [sloped,below] {\tiny $d_2-d_1$} (2);
\path[->,very thick,color=blue] (3) edge node [sloped,above left] {\tiny $d_1-d_3$} (1);
\end{tikzpicture}
&
\begin{tikzpicture}[scale=0.5]
\node[draw] at (0,0) (0) {\tiny $V_0$};
\node[draw] at (0,-4) (1) {\tiny $V_1$};
\node[draw] at (4,0) (2) {\tiny $V_2$};
\node[draw] at (4,2) (3) {\tiny $V_3$};

\path[->,dashed] (0) edge node [left] {\tiny $d_1$} (1);
\path[->,dashed] (0) edge node [above] {\tiny $d_2$} (2);
\path[->,very thick,color=blue] (0) edge node [sloped,above] {\tiny $d_3$} (3);
\path[->,dashed] (2) edge node [right] {\tiny $d_3-d_2$} (3);
\path[->,very thick,color=blue] (1) edge node [sloped,below] {\tiny $d_2-d_1$} (2);
\path[->,dashed] (3) edge node [sloped,above left] {\tiny $d_1-d_3$} (1);
\end{tikzpicture} \\ \ \\
$d_1$ and $d_3-d_2$ & $d_2$ and $d_1-d_3$ & $d_3$ and $d_2-d_1$ \\
\end{tabular}
\caption{Non-adjacent common differences of $\AS(a,(d_1,d_2,d_3),s)$}\label{fig3}
\end{figure}

\begin{theorem}\label{thmarithdim2}
Let $m$ and $s$ be two positive integers and let $a,d_1,d_2,d_3\in\ZZ{m}$. Let $D := \left\{ d_1,d_2,d_3,d_2-d_1,d_3-d_2,d_1-d_3 \right\}$ be the set of common differences of the arithmetic tetrahedron $\triangle=\AS(a,(d_1,d_2,d_3),s)$. If $\triangle$ is balanced in $\ZZ{m}$ and
\begin{enumerate}[i)]
\item\label{casei}
$m$ is odd, then all the elements of $D$ are invertible.
\item\label{caseii}
$m$ is even, then all the elements of $D$ are invertible, except two of them, say $\delta_1$ and $\delta_2$, which are non-adjacent and such that $\gcd(\delta_1,m)=\gcd(\delta_2,m)=2$.
\end{enumerate}
\end{theorem}

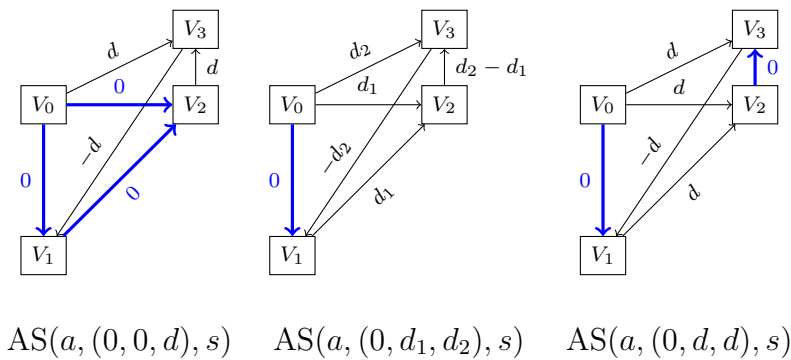
\begin{figure}
\begin{tabular}{ccc}
\begin{tikzpicture}[scale=0.5]
\node[draw] at (0,0) (0) {\tiny $V_0$};
\node[draw] at (0,-4) (1) {\tiny $V_1$};
\node[draw] at (4,0) (2) {\tiny $V_2$};
\node[draw] at (4,2) (3) {\tiny $V_3$};

\path[->,very thick,color=blue] (0) edge node [left] {\tiny $0$} (1);
\path[->,very thick,color=blue] (0) edge node [above] {\tiny $0$} (2);
\path[->] (0) edge node [sloped,above] {\tiny $d$} (3);
\path[->] (2) edge node [right] {\tiny $d$} (3);
\path[->,very thick,color=blue] (1) edge node [sloped,below] {\tiny $0$} (2);
\path[->] (3) edge node [sloped,above left] {\tiny $-d$} (1);
\end{tikzpicture}
&
\begin{tikzpicture}[scale=0.5]
\node[draw] at (0,0) (0) {\tiny $V_0$};
\node[draw] at (0,-4) (1) {\tiny $V_1$};
\node[draw] at (4,0) (2) {\tiny $V_2$};
\node[draw] at (4,2) (3) {\tiny $V_3$};

\path[->,very thick,color=blue] (0) edge node [left] {\tiny $0$} (1);
\path[->] (0) edge node [above] {\tiny $d_1$} (2);
\path[->] (0) edge node [sloped,above] {\tiny $d_2$} (3);
\path[->] (2) edge node [right] {\tiny $d_2-d_1$} (3);
\path[->] (1) edge node [sloped,below] {\tiny $d_1$} (2);
\path[->] (3) edge node [sloped,above left] {\tiny $-d_2$} (1);
\end{tikzpicture}
&
\begin{tikzpicture}[scale=0.5]
\node[draw] at (0,0) (0) {\tiny $V_0$};
\node[draw] at (0,-4) (1) {\tiny $V_1$};
\node[draw] at (4,0) (2) {\tiny $V_2$};
\node[draw] at (4,2) (3) {\tiny $V_3$};

\path[->,very thick,color=blue] (0) edge node [left] {\tiny $0$} (1);
\path[->] (0) edge node [above] {\tiny $d$} (2);
\path[->] (0) edge node [sloped,above] {\tiny $d$} (3);
\path[->,very thick,color=blue] (2) edge node [right] {\tiny $0$} (3);
\path[->] (1) edge node [sloped,below] {\tiny $d$} (2);
\path[->] (3) edge node [sloped,above left] {\tiny $-d$} (1);
\end{tikzpicture} \\ \ \\
$\AS(a,(0,0,d),s)$ & $\AS(a,(0,d_1,d_2),s)$ & $\AS(a,(0,d,d),s)$ \\
\end{tabular}
\caption{Arithmetic tetrahedra with a common difference equal to zero}\label{fig6}
\end{figure}

The proof of this theorem is based on the following four lemmas, where we study the multiplicity function of the arithmetic tetrahedra with at least one common difference equal to zero. A representation of the common differences of these arithmetic tetrahedra can be found in Figure~\ref{fig6} and Figure~\ref{fig7}.

\begin{lemma}\label{prop8}
Let $m$ and $s$ be two positive integers and let $a,d\in\ZZ{m}$ such that $d$ is invertible. Then, the arithmetic tetrahedron $\AS(a,(0,0,d),s)$ is not balanced in $\ZZ{m}$.
\end{lemma}

\begin{proof}
Let $\triangle=\AS(a,(0,0,d),s)$. By Lemma~\ref{keylemma}, we obtain
$$
\m_{\triangle}(a) - \m_{\triangle}(a+d) = \m_{\F_1(\triangle)}(a) - \m_{\F_3(\triangle)}(a+d).
$$
Since $\F_3(\triangle) = \AS(a,(0,0),s)$, i.e., the constant triangle of size $s$ where all the terms are equal to $a$, we have $\m_{\F_3(\triangle)}(a+d)=0$. Since $\F_1(\triangle)=\AS(a,(0,d),s)$, it follows from Lemma~\ref{prop1} that
$$
\m_{\F_1(\triangle)}(a) = \binom{\lambda+1}{2}m + \left\lceil\frac{s-i}{m}\right\rceil\mu > 0,
$$
where $s=\lambda m + \mu$ is the Euclidean division of $s$ by $m$. This leads to the inequality $\m_{\triangle}(a+d) < \m_{\triangle}(a)$ and thus the multiplicity function $\m_{\triangle}$ is not constant on $\ZZ{m}$. This concludes the proof.
\end{proof}

\begin{lemma}\label{prop2}
Let $m$ and $s$ be two positive integers and let $a,d_1,d_2\in\ZZ{m}$ such that $d_1$, $d_2$ and $d_2-d_1$ are invertible. Then, the arithmetic tetrahedron $\AS(a,(0,d_1,d_2),s)$ is not balanced in $\ZZ{m}$.
\end{lemma}

\begin{proof}
First, since $d_1$, $d_2$ and $d_2-d_1$ are invertible, we know that $m$ is an odd number. Assume, without loss of generality, that $m$ is an odd prime. Indeed, if $m$ is composite, we know that $\triangle$ cannot be balanced in $\ZZ{m}$ if its projected tetrahedron $\pi_p(\triangle)$ is not balanced in $\ZZ{p}$, where $p$ is an odd prime factor of $m$. Moreover, a necessary condition on $s$ for $\triangle$ being balanced is that $m$ divides $\binom{s+2}{3}$. Since $m$ is an odd prime, this implies that $s\equiv 0,-1$ or $-2 \bmod{m}$. If $s\equiv -2\bmod{m}$, then $\triangle=\AS(a,(0,d_1,d_2),s)$ can be seen as $\triangle'\setminus \F_1(\triangle')$, where $\triangle'=\AS(a,(0,d_1,d_2),s+1)$ and $\F_1(\triangle')=\AS(a,(d_1,d_2),s+1)$. Since $s+1\equiv-1\bmod{m}$ and $d_1$, $d_2$, $d_2-d_1$ are invertible, we deduce from Theorem~\ref{thm2} that $\F_1(\triangle')$ is balanced. Therefore $\triangle$ is balanced if and only if $\triangle'$ is. This is the reason why, in this proof, we suppose that $m$ is an odd prime and that $s$ is a positive integer such that $s\equiv 0$ or $-1\bmod{m}$. First, by Lemma~\ref{keylemma}, we obtain
$$
\m_{\triangle}(x+d_2)-\m_{\triangle}(x) = \m_{\F_3(\triangle)}(x+d_2) - \m_{\F_1(\triangle)}(x),
$$
for all $x\in\ZZ{m}$. Since $d_1$, $d_2$ and $d_2-d_1$ are invertible and $s\equiv0$ or $-1\bmod{m}$, we know, from Theorem~\ref{thm2}, that $\F_1(\triangle)=\AS(a,(d_1,d_2),s)$ is a balanced triangle. Therefore,
$$
\m_{\F_1(\triangle)}(x) = \frac{1}{m}\binom{s+1}{2},
$$
for all $x\in\ZZ{m}$. Moreover, since $d_1$ is invertible, it follows from Lemma~\ref{prop1} that the multiplicity function of $\F_3(\triangle)=\AS(a,(0,d_1),s)$ verifies
$$
\m_{\F_3(\triangle)}(a) > \m_{\F_3(\triangle)}(a+d_1) > \cdots\cdots > \m_{\F_3(\triangle)}(a+(m-1)d_1).
$$
Therefore, since $\sum_{i=0}^{m-1}\m_{\F_3(\triangle)}(a+id)=\binom{s+1}{2}$, we have $\m_{\F_3(\triangle)}(a) > \frac{1}{m}\binom{s+1}{2}$. This leads to the inequality
$$
\m_{\triangle}(a) - \m_{\triangle}(a-d_2) = \m_{\F_3(\triangle)}(a) - \m_{\F_1(\triangle)}(a-d_2) = \m_{\F_3(\triangle)}(a) - \frac{1}{m}\binom{s+1}{2} > 0.
$$
Thus, the multiplicity function $\m_{\triangle}$ is not constant on $\ZZ{m}$. This concludes the proof.
\end{proof}

\begin{lemma}\label{prop12}
Let $m$ and $s$ be two positive integers and let $a,d\in\ZZ{m}$ such that $d$ is invertible. If the arithmetic tetrahedron $\mathrm{AS}(a,(0,d,d),s)$ is balanced, then $m=2$ and $s$ is even.
\end{lemma}

\begin{proof}
Let $\triangle=\AS(a,(0,d,d),s)$. First, by Lemma~\ref{keylemma}, we obtain the identity
$$
\m_{\triangle}(x+d) - \m_{\triangle}(x) = \m_{\F_2(\triangle)}(x+d) - \m_{\F_1(\triangle)}(x),
$$
for all $x\in\ZZ{m}$. Let $s=\lambda m+\mu$ be the Euclidean division of $s$ by $m$. Since $\F_2(\triangle)=\AS(a,(0,d),s)$, we know, from Lemma~\ref{prop1}, that
$$
\m_{\F_2(\triangle)}(a+id) = \binom{\lambda+1}{2}m + \left\lceil\frac{s-i}{m}\right\rceil (\mu-i)
$$
for all $i\in[0,m-1]$. Moreover, since $\F_1(\triangle)=\AS(a,(d,d),s)=\AS(a+(s-1)d,(0,-d),s)$, by Lemma~\ref{prop1} again, we have
$$
\m_{\F_1(\triangle)}(a+(s-1-j)d) = \binom{\lambda+1}{2}m + \left\lceil\frac{s-j}{m}\right\rceil (\mu-j)
$$
for all $j\in[0,m-1]$. This leads to the identity
$$
\m_{\triangle}(a) - \m_{\triangle}(a-d)
\begin{array}[t]{l}
 = \m_{\F_2(\triangle)}(a) - \m_{\F_1(\triangle)}(a-d) \\
 \ \\
 = \m_{\F_2(\triangle)}(a) - \m_{\F_1(\triangle)}(a+(s-1-\mu)d) = \displaystyle\left\lceil\frac{s}{m}\right\rceil \mu.
\end{array}
$$
If $\triangle$ is balanced, then $\m_{\triangle}$ is constant on $\ZZ{m}$. Then, $\m_{\triangle}(a)=\m_{\triangle}(a-d)$ and $\left\lceil\frac{s}{m}\right\rceil\mu=0$. Therefore $\mu=0$ and $s\equiv0\bmod{m}$. Suppose now that $s=\lambda m$, with $\lambda\ge1$. For every integer $i\in[1,m-1]$, we deduce from the previous results that
$$
\m_{\triangle}(a+id) - \m_{\triangle}(a+(i-1)d) 
\begin{array}[t]{l}
 = \m_{\F_2(\triangle)}(a+id) - \m_{\F_1(\triangle)}(a+(i-1)d) \\
 \ \\
 = \m_{\F_2(\triangle)}(a+id) - \m_{\F_1(\triangle)}(a+(s-1-(m-i))d) \\
 \ \\
 = \lambda(m-2i).
\end{array}
$$ 
The only possibility for having $\m_{\triangle}(a+id) = \m_{\triangle}(a+(i-1)d)$ is that $m=2$ and $i=1$. This concludes the proof.
\end{proof}

\begin{figure}
\begin{tikzpicture}[scale=0.5]
\node[draw] at (0,0) (0) {\tiny $V_0$};
\node[draw] at (0,-4) (1) {\tiny $V_1$};
\node[draw] at (4,0) (2) {\tiny $V_2$};
\node[draw] at (4,2) (3) {\tiny $V_3$};

\path[->,very thick,color=blue] (0) edge node [left] {\tiny $0$} (1);
\path[->] (0) edge node [above] {\tiny $d$} (2);
\path[->] (0) edge node [sloped,above] {\tiny $-d$} (3);
\path[->,very thick,color=blue] (2) edge node [right] {\tiny $2$} (3);
\path[->] (1) edge node [sloped,below] {\tiny $d$} (2);
\path[->] (3) edge node [sloped,above left] {\tiny $d$} (1);
\end{tikzpicture}
\caption{$\AS(a,(0,d,-d),s)$ in $\ZZ{4}$, with $d=\pm1$}\label{fig7}
\end{figure}

\begin{lemma}\label{prop3}
Let $s$ be a positive integer and let $a,d\in\ZZ{4}$ such that $d$ is invertible. Then, the arithmetic tetrahedron $\AS(a,(0,d,-d),s)$ is not balanced in $\ZZ{4}$.
\end{lemma}

\begin{proof}
Let $\triangle=\AS(a,(0,d,-d),s)$. If $\triangle$ is balanced in $\ZZ{4}$, then its projection $\pi_2(\triangle) = \AS( \pi_2(a) , (0,1,1) , s )$ is balanced in $\ZZ{2}$. This implies, by Lemma~\ref{prop12}, that $s$ is even. Let $s=4\lambda + \mu$ be the Euclidean division of $s$ by $4$, where $\mu\in\{0,2\}$. Since $\F_3(\triangle) = \AS(a,(0,d),s)$ and $\F_2(\triangle) = \AS(a,(0,-d),s)$, it follows from Lemma~\ref{prop1} and Lemma~\ref{keylemma} that
$$
\m_{\triangle}(a) - \m_{\triangle}(a+2)
\begin{array}[t]{l}
= \m_{\F_3(\triangle)}(a) - \m_{\F_2(\triangle)}(a+2) \\
\ \\
= \m_{\F_3(\triangle)}(a) - \m_{\F_2(\triangle)}(a-2d) \\
\ \\
= \displaystyle 4\binom{\lambda+1}{2} + \left\lceil\frac{4\lambda+\mu}{4}\right\rceil\mu - 4\binom{\lambda+1}{2} - \left\lceil\frac{4\lambda+\mu-2}{4}\right\rceil(\mu-2) \\
\ \\
= \displaystyle \left\lceil\frac{4\lambda+\mu}{4}\right\rceil\mu - \left\lceil\frac{4\lambda+\mu-2}{4}\right\rceil(\mu-2).
\end{array}
$$
For $\mu=0$, we obtain that $\m_{\triangle}(a) - \m_{\triangle}(a+2) = 2\lambda \neq 0$. For $\mu=2$, $\m_{\triangle}(a) - \m_{\triangle}(a+2) = 2(\lambda+1) \ge 2$. In all cases, we obtain that $\m_{\triangle}(a)\neq\m_{\triangle}(a+2)$. Therefore the tetrahedron $\triangle$ is not balanced in $\ZZ{4}$.
\end{proof}

We are now ready to prove Theorem~\ref{thmarithdim2}.

\begin{proof}[Proof of Theorem~\ref{thmarithdim2}]
Let $\triangle=\AS(a,(d_1,d_2,d_3),s)$ be an arithmetic tetrahedron of size $s$, where all the common differences are
$$
D := \{d_1,d_2,d_3,d_2-d_1,d_3-d_2,d_1-d_3\}.
$$
If there exists an element $\delta$ of $D$ which is not invertible in $\ZZ{m}$ and not equal to zero, then we consider the projection of $\triangle$ into $\ZZ{\alpha}$, where $\alpha=\gcd(\delta,m)\ge 2$. The projected tetrahedron $\pi_{\alpha}(\triangle)$ is also arithmetic. In this tetrahedron, the corresponding common differences of the invertible common differences of $\triangle$ are invertible in $\ZZ{\alpha}$ and $\pi_{\alpha}(\delta)=0$. If there exists a common difference $\delta'$ of $\pi_{\alpha}(D)$ which is not invertible in $\ZZ{\alpha}$ and not equal to zero, we project again $\pi_{\alpha}(\triangle)$ into $\ZZ{\beta}$, where $\beta=\gcd(\delta',\alpha)\ge 2$. We continue until the projected tetrahedron is such that all its common differences are either invertible, or equal to zero. In the sequel, suppose that $\pi_{\alpha}(\triangle)$ is like that. Obviously, if the projected tetrahedron $\pi_{\alpha}(\triangle)$ is not balanced in $\ZZ{\alpha}$, we know that $\triangle$ cannot be balanced in $\ZZ{m}$. We distinguish different cases.
\setcounter{case}{0}
\begin{case}\label{c1}
There exist three adjacent common differences, in $\pi_{\alpha}(\triangle)$, which are equal to zero. Then, all the elements of $\pi_{\alpha}(D)$ are equal to zero and $\pi_{\alpha}(\triangle)$ is the constant tetrahedron where all terms are equal to $\pi_{\alpha}(a)$. Thus, the tetrahedron $\pi_{\alpha}(\triangle)$ is not balanced in $\ZZ{\alpha}$.
\end{case}
\begin{case}\label{c2}
There exist two adjacent common differences, in $\pi_{\alpha}(\triangle)$, which are equal to zero. Without loss of generality, suppose that $\pi_{\alpha}(d_1)=\pi_{\alpha}(d_2)=0$. Then, $\pi_{\alpha}(\triangle)=\AS(\pi_{\alpha}(a),(0,0,\pi_{\alpha}(d_3)),s)$. If $\pi_{\alpha}(d_3)=0$, this is Case~\ref{c1}. Otherwise, if $\pi_{\alpha}(d_3)$ is invertible, the tetrahedron $\pi_{\alpha}(\triangle)$ is not balanced in $\ZZ{\alpha}$, by Lemma~\ref{prop8}.
\end{case}
\begin{case}\label{c3}
There exist two non-adjacent common differences, in $\pi_{\alpha}(\triangle)$, which are equal to zero. Without loss of generality, suppose that $\pi_{\alpha}(d_1)=\pi_{\alpha}(d_3-d_2)=0$. Then, $\pi_{\alpha}(\triangle)=\AS(\pi_{\alpha}(a),(0,\pi_{\alpha}(d_2),\pi_{\alpha}(d_2)),s)$. If $\pi_{\alpha}(d_2)=0$, this is Case~\ref{c1}. Otherwise, if $\pi_{\alpha}(d_2)$ is invertible, we know from Lemma~\ref{prop12} that if $\pi_{\alpha}(\triangle)$ is balanced, then $\alpha=2$ and $s$ is even.
\end{case}
\begin{case}
There exists one common difference, in $\pi_{\alpha}(\triangle)$, which is equal to zero. Without loss of generality, suppose that $\pi_{\alpha}(d_1)=0$. If there exists a second common difference which is equal to zero, this is Case~\ref{c2} or Case~\ref{c3}, depending on the adjacency of these two common differences. If the five other common differences of $\pi_{\alpha}(\triangle)$ are invertible, we know that $\pi_{\alpha}(\triangle)$ is not balanced by Lemma~\ref{prop2}.
\end{case}
Therefore, the only possibility for $\pi_{\alpha}(\triangle)$ to be balanced in $\ZZ{\alpha}$ is that $\alpha=2$ and $s$ is even. Thus, when $m$ is odd, we have proved the result~\ref{casei}) of Theorem~\ref{thmarithdim2}. When $m$ is even, we deduce from the previous results that if $\triangle$ is balanced with a non-invertible common difference $\delta$, then the corresponding non-adjacent common difference of $\delta$ is also non-invertible and their projections in $\ZZ{2}$ are equal to zero. Moreover, since the constant tetrahedra in $\ZZ{2}$ are not balanced, it follows that the four other common differences of $\triangle$ must be invertible in $\ZZ{m}$. Suppose now that $m$ and $s$ are even and that the common differences of the balanced tetrahedron $\triangle$ are such that $d_1$ and $d_3-d_2$ are non-invertible and $d_2$, $d_3$, $d_2-d_1$, $d_1-d_3$ are invertible, where
$$
\pi_2(\triangle)=\AS(\pi_2(a),(0,1,1),s).
$$
If $\gcd(d_1,m)$ or $\gcd(d_3-d_2,m)$ is divisible by an odd number $m'$, then the projected tetrahedron $\pi_{m'}(\triangle)$ is a balanced arithmetic tetrahedron in $\ZZ{m'}$, with at least one common difference which is non-invertible, which is impossible by result~\ref{casei}) of Theorem~\ref{thmarithdim2}. Thus $\gcd(d_1,m)$ and $\gcd(d_3-d_2,m)$ are powers of two. Finally, from Lemma~\ref{prop3}, we deduce that $\gcd(d_1,m)=\gcd(d_3-d_2,m)=2$. This completes the proof.
\end{proof}

We continue by showing that there is no balanced arithmetic tetrahedron in $\ZZ{m}$ when $m$ is divisible by $3$.

\begin{theorem}
Let $m$ and $s$ be two positive integers such that $m$ is a multiple of $3$. There is no balanced arithmetic tetrahedron of size $s$ in $\ZZ{m}$.
\end{theorem}

\begin{proof}
Let $\triangle$ be an arithmetic tetrahedron of size $s$ in $\ZZ{m}$. We consider the projected tetrahedron $\pi_3(\triangle)$ in $\ZZ{3}$. We know from Theorem~\ref{thmarithdim2} that if $\pi_3(\triangle)$ is a balanced tetrahedron, then all the common differences of $\pi_3(\triangle)$ are invertible in $\ZZ{3}$. In $\ZZ{3}$, the invertible elements are $1$ and $2$. Thus, if $d_1$, $d_2$ and $d_3$ are invertible, there are at least two of them which are equal and thus their difference is a common difference equal to zero. Therefore $\pi_3(\triangle)$ cannot be balanced in $\ZZ{3}$. This completes the proof.
\end{proof}

In the end of this subsection, we prove that the necessary conditions on the common differences of balanced arithmetic tetrahedra highlighted in Theorem~\ref{thmarithdim2} are also sufficient for certain sizes.

\begin{theorem}
Let $m$ be an odd number not divisible by $3$ and let $a,d_1,d_2,d_3\in\ZZ{m}$ such that $d_1$, $d_2$, $d_3$, $d_2-d_1$, $d_3-d_2$ and $d_1-d_3$ are invertible. Then, the arithmetic tetrahedron $\AS(a,(d_1,d_2,d_3),s)$ is balanced for all $s\equiv 0$, $-1$, or $-2\bmod{m}$.
\end{theorem}

\begin{proof}
Theorem~\ref{thm2} for $n=3$.
\end{proof}

\begin{theorem}\label{thm5}
Let $m$ be an even number not divisible by $3$ and let $a,d_1,d_2,d_3\in\ZZ{m}$ such that $\gcd(d_1,m)=\gcd(d_3-d_2,m)=2$ and $d_2$, $d_3$, $d_2-d_1$ and $d_1-d_3$ are invertible. Then, the arithmetic tetrahedron $\AS(a,(d_1,d_2,d_3),s)$ is balanced for all $s\equiv 0$ or $-2\bmod{m}$.
\end{theorem}

\begin{proof}
Let $\triangle=\AS(a,(d_1,d_2,d_3),s)$.
\setcounter{case}{0}
\begin{case}
Suppose that $s$ is divisible by $m$.

Since $\gcd(d_1,m)=2$, $d_3$ and $d_1-d_3$ invertible and $s$ divisible by $m$, it follows from Proposition~\ref{propbalasn2} that the multiplicity function of $\F_2(\triangle)=\AS(a,(d_1,d_3),s)$ is entirely determined by
$$
\m_{\F_2(\triangle)}(a+2i) = \frac{1}{m}\binom{s+1}{2} + \frac{s}{2m},\quad \m_{\F_2(\triangle)}(a+2i+1) = \frac{1}{m}\binom{s+1}{2} - \frac{s}{2m},
$$
for all $i\in[0,\frac{m}{2}-1]$. Similarly, since $\gcd(d_3-d_2,m)=2$, $d_2$ and $d_3$ invertible and $s$ divisible by $m$, we have by Proposition~\ref{propbalasn2} that the multiplicity function of $\F_1(\triangle)=\AS(a,(d_2,d_3),s)=\AS(a+(s-1)d_2,d_3-d_2,-d_2,s)$ is equal to
$$
\m_{\F_1(\triangle)}(a+2i+1) = \frac{1}{m}\binom{s+1}{2} + \frac{s}{2m},\quad \m_{\F_1(\triangle)}(a+2i) = \frac{1}{m}\binom{s+1}{2} - \frac{s}{2m},
$$
for all $i\in[0,\frac{m}{2}-1]$. This leads to the identity
$$
\m_{\F_2(\triangle)}(x+1) = \m_{\F_1(\triangle)}(x),
$$
for all $x\in\ZZ{m}$. Then, since $\gcd(d_1,m)=2$ and $d_2$ is invertible, we obtain by Lemma~\ref{keylemma} that
$$
\m_{\triangle}(x+d_1) - \m_{\triangle}(x+d_2)
\begin{array}[t]{l}
= \m_{\F_1(\triangle)}(x+d_1) - \m_{\F_2(\triangle)}(x+d_2) \\
= \m_{\F_1(\triangle)}(x+2) - \m_{\F_2(\triangle)}(x+1) = 0, \\
\end{array}
$$
and thus
$$
\m_{\triangle}(x+(d_1-d_2))=\m_{\triangle}(x)
$$
for all $x\in\ZZ{m}$. Since $d_1-d_2$ is invertible, we conclude that $\m_{\triangle}$ is constant and $\triangle$ is balanced in $\ZZ{m}$.
\end{case}
\begin{case}
Now, suppose that $s\equiv -2\bmod{m}$.

The tetrahedron $\triangle$ can be seen as the arithmetic tetrahedron $\triangle'$ of size $s+2$ where the first two rows have been removed, i.e., $\triangle=\triangle'\setminus\{\R_0,\R_1\}$, where $\triangle' := \AS(a-2d_3,(d_1,d_2,d_3),s+2)$, $\R_0 := \AS(a-2d_3,(d_1,d_2),s+2)$ and $\R_1:= \AS(a-d_3,(d_1,d_2),s+1)$. Since $\gcd(d_1,m)=2$, $d_2$, $d_2-d_1$, $d_3$ invertible and $s+2\equiv 0$, $s+1\equiv -1\bmod{m}$, it follows from Proposition~\ref{propbalasn2} again that
$$
\begin{array}{l}
\displaystyle\m_{\R_0}(a+2i) = \frac{1}{m}\binom{s+3}{2} + \frac{s+2}{2m},\quad \m_{\R_0}(a+2i+1) = \frac{1}{m}\binom{s+3}{2} - \frac{s+2}{2m}, \\
\ \\
\displaystyle\m_{\R_1}(a+2i+1) = \frac{1}{m}\binom{s+2}{2} + \frac{s+2}{2m},\quad \m_{\R_1}(a+2i) = \frac{1}{m}\binom{s+2}{2} - \frac{s+2}{2m},
\end{array}
$$
for all $i\in[0,\frac{m}{2}-1]$. This leads to
$$
\m_{\R_0}(x)+\m_{\R_1} (x) = \frac{1}{m}\left(\binom{s+3}{2}+\binom{s+2}{2}\right),
$$
for all $x\in\ZZ{m}$. Therefore the multiset $\R_0\cup \R_1$ is balanced in $\ZZ{m}$. Moreover, since $s+2\equiv 0\bmod{m}$, we already know from Case~1 that $\triangle'$ is balanced. The multiset difference of balanced multisets is obviously balanced. This completes the proof.
\end{case}
\vspace{-15pt}
\end{proof}

\begin{remark}
When $m$ is even not divisible by $3$, if we suppose that $\gcd(d_1,m)=\gcd(d_3-d_2,m)=2$ and $d_2$, $d_3$, $d_2-d_1$ and $d_1-d_3$ are invertible, then the arithmetic tetrahedron $\AS(a,(d_1,d_2,d_3),s)$ is not balanced for all $s\equiv -1\bmod{m}$. Indeed, it can be seen as the multiset difference of the arithmetic tetrahedron $\AS(a-d_3,(d_1,d_2,d_3),s+1)$, which is balanced by Theorem~\ref{thm5}, and the arithmetic triangle $\AS(a-d_3,(d_1,d_2),s+1)$, which is not balanced by Theorem~\ref{balasn2}.
\end{remark}

\section{Balanced simplices generated from arithmetic arrays}

We are now ready to show that the orbits generated from arithmetic arrays by additive cellular automata are a source of balanced simplices.

\subsection{The general case: in dimension $\mathbf{n\ge2}$}

In this subsection, we prove Theorem~\ref{thm1}, the main result of this paper, and, in corollary, the special case of the Pascal cellular automata is examined.

\subsubsection{For any ACA}

The proof of Theorem~\ref{thm1} is based on the following lemma.

\begin{lemma}\label{lem1}
Let $n\ge2$, $m_1$ and $m_2$ be three positive integers such that $\gcd(m_2,n!)=1$ and let $m:=m_1m_2$. Let $a\in\ZZ{m}$ and let $d=(d_1,\ldots,d_n)\in(\ZZ{m})^{n}$ such that $\pi_{m_2}(d_i)$, for all $i\in[1,n]$, and $\pi_{m_2}(d_j)-\pi_{m_2}(d_i)$, for all distinct integers $i,j\in[1,n]$, are invertible in $\ZZ{m_2}$. Then, for all positive integers $s\equiv -t\bmod{m_2}$, where $t\in[0,n-1]$, the multiplicity function of the arithmetic simplex $\triangle=\AS(a,m_1d,s)$ of $\ZZ{m}$ verifies
$$
\m_{\triangle}(x+m_1) = \m_{\triangle}(x),
$$
for all $x\in\ZZ{m}$.
\end{lemma}

\begin{proof}
First, it is clear that $\triangle$ is uniquely composed of elements of the form $a+km_1$, where $k\in[0,m_2-1]$. Therefore, the identity is obviously true for all elements $x$ not in $\left\{ a+km_1\ \middle|\ k\in[0,m_2-1]\right\}$. Moreover, since $\pi_{m_2}(d_i)$, for all $i\in[1,n]$, and $\pi_{m_2}(d_j)-\pi_{m_2}(d_i)$, for all distinct integers $i,j\in[1,n]$, are invertible in $\ZZ{m_2}$ and since $s\equiv -t\bmod{m_2}$, where $t\in[0,n-1]$, we know from Theorem~\ref{thm2} that the arithmetic simplex $\AS(0,\pi_2(d),s)$ is balanced in $\ZZ{m_2}$. Finally, since $\triangle$ can be seen as the image of the balanced arithmetic simplex $\triangle'=\AS(0,\pi_2(d),s)$ by the function
$$
f : 
\begin{array}[t]{ccc}
 \ZZ{m_2} & \longrightarrow & \ZZ{m} \\
 \pi_{m_2}(x) & \longmapsto & a+m_1x
\end{array}
$$
it follows that
$$
\m_{\triangle}(a+km_1) = \m_{f(\triangle')}(f(k)) = \m_{\triangle'}(k) = \m_{\triangle'}(0) = \m_{f(\triangle')}(f(0)) = \m_{\triangle}(a),
$$
for all $k\in[0,m_2-1]$. This completes the proof.
\end{proof}

We are now ready to prove Theorem~\ref{thm1}.

\begin{proof}[Proof of Theorem~\ref{thm1}]
Let $\triangle(j,\varepsilon,s)$ be an $n$-simplex of size $s=\lambda\lcm(\ord_m(\sigma),m)-t$, where $t\in[0,n-1]$, appearing in the orbit $\orb(\ArA(a,d))=(a_i)_{i\in\Z^{n-1}\times\N}$. We proceed by induction on $m$. For $m=1$, the result is obvious. Suppose now that the result is true for all finite cyclic groups of order strictly less than $m$. Let
$$
m_1 := \gcd\left(\ord_m(\sigma),m\right)\quad\text{and}\quad m_2:=\frac{m}{m_1}.
$$
Then, $s=\lambda\ord_m(\sigma)m_2-t$.

First, we prove that $\m_{\triangle}(x+m_1)=\m_{\triangle}(x)$ for all $x\in\ZZ{m}$. By Proposition~\ref{propsub} for $\alpha=\ord_m(\sigma)$, we know that $\triangle(j,\varepsilon,s)$ can be decomposed into ${\ord_m(\sigma)}^{n}$ subsimplices $\SuS_k$,
$$
\triangle(j,\varepsilon,s) = \bigcup_{k\in{[0,\ord_m(\sigma)-1]}^{n}}{\SuS_k}
$$
that are, for all $k\in{[0,\ord_m(\sigma)-1]}^{n}$, the arithmetic simplices
$$
\SuS_k = \AS\left(a_{j+\varepsilon\cdot k},\ord_m(\sigma)\sigma^{j_n+\varepsilon_n k_n}\varepsilon\cdot \tilde{d},\lambda m_2-\left\lfloor\frac{\sum_{u=1}^{n}k_u+t}{\ord_m(\sigma)}\right\rfloor\right),
$$
where $\tilde{d}=(d_1,\ldots,d_{n-1},d_n)$ and $d_n=\sigma^{-1}\sum_{u=1}^{n-1}\sigma_ud_u$. Since
$$
\gcd\left( \frac{\ord_m(\sigma)}{m_1} , m_2 \right) = \frac{\gcd(\ord_m(\sigma),m)}{\gcd(\ord_m(\sigma),m)} = 1,
$$
it follows, from the hypothesis of the theorem, that the elements
$$
\pi_{m_2}\left(\frac{\ord_m(\sigma)}{m_1}\sigma^{j_n+\varepsilon_nk_n}\varepsilon_u d_u\right),
$$
for all $u\in[1,n]$, and
$$
\pi_{m_2}\left(\frac{\ord_m(\sigma)}{m_1}\sigma^{j_n+\varepsilon_nk_n}(\varepsilon_u d_u-\varepsilon_v d_v)\right),
$$
for all distinct integers $u,v\in[1,n]$, are invertible in $\ZZ{m_2}$. Moreover, since the size of $\SuS_k$, $\lambda m_2-\left\lfloor\frac{\sum_{u=1}^{n}k_u+t}{\ord_m(\sigma)}\right\rfloor$, is in $[\lambda m_2-(n-1),\lambda m_2]$ and thus is congruent to a certain integer $-y$ modulo $m_2$ with $y\in[0,n-1]$, we deduce from Lemma~\ref{lem1} that
$$
\m_{\SuS_k}(x+m_1) = \m_{\SuS_k}(x),
$$
for all $x\in\ZZ{m}$, and this, for all $k\in{[0,\ord_m(\sigma)-1]}^{n}$. Then, the multiplicity function of $\triangle$ verifies
$$
\m_{\triangle}(x+m_1) = \hspace{-10pt}\sum_{k\in{[0,\ord_m(\sigma)-1]}^{n}}\hspace{-10pt}\m_{\SuS_k}(x+m_1) = \hspace{-10pt}\sum_{k\in{[0,\ord_m(\sigma)-1]}^{n}}\hspace{-10pt}\m_{\SuS_k}(x) = \m_{\triangle}(x),
$$
for all $x\in\ZZ{m}$.

Now, we prove that the projected simplex $\pi_{m_1}(\triangle)$ is balanced. We begin by showing that the integer $\ord_m(\sigma^m)m$ is divisible by $\ord_{m_1}(\sigma^{m_1})m_1$. Indeed,
$$
\left(\sigma^{m_1}\right)^{m_2\ord_m(\sigma^m)} = \left(\sigma^{m}\right)^{\ord_m(\sigma^m)} \equiv 1 \pmod{m}
$$
and, by divisibility of $m$ by $m_1$, this equivalence is also true modulo $m_1$. It follows that $m_2\ord_{m}(\sigma^{m})$ is divisible by $\ord_{m_1}(\sigma^{m_1})$, implying that $\ord_m(\sigma^m)m$ is divisible by $\ord_{m_1}(\sigma^{m_1})m_1$. Then, $s\equiv -t\bmod{\ord_{m_1}(\sigma^{m_1})m_1}$ and since the elements $\pi_{m_1}(d_u)$, for all $u\in[1,n]$, and $\pi_{m_1}(\varepsilon_vd_v-\varepsilon_ud_u)$, for all distinct integers $u,v\in[1,n]$, are clearly invertible in $\ZZ{m_1}$, we deduce from the induction hypothesis that $\pi_{m_1}(\triangle)$ is balanced in $\ZZ{m_1}$.

Finally, since $\pi_{m_1}(\triangle)$ is balanced and $\m_{\triangle}(x+m_1)=\m_{\triangle}(x)$ for all $x\in\ZZ{m}$, we deduce from Theorem~\ref{thmproj} that the simplex $\triangle$ is balanced. This concludes the proof.
\end{proof}

\begin{remark}
If $\sigma\equiv 1\bmod{m}$, then the $n$-simplex $\triangle\left(j,\varepsilon,s\right)$ appearing in $\orb(\ArA(a,d))$ is an arithmetic simplex and Theorem~\ref{thm1} is simply Theorem~\ref{thm2} on balanced arithmetic simplices.
\end{remark}

\subsubsection{For the Pascal cellular automata}

Here, we investigate the consequences of Theorem~\ref{thm1} on the existence of balanced $n$-simplices, in the case where the ACA considered is $\PCA{n-1}$.

\begin{corollary}
Let $n\ge 2$ be a positive integers. For every positive integer $m$ such that $\gcd(m,(3(n-1))!)=1$, there exist infinitely many balanced $n$-simplices of $\ZZ{m}$ generated by $\PCA{n-1}$, for all possible orientations $\varepsilon\in\{-1,1\}^ n$. In the special case of the two orientations $\varepsilon=+\cdots+-$ or $\varepsilon=-\cdots-+$, the existence of an infinite number of such balanced simplices is verified for every $\ZZ{m}$ such that $\gcd(m,n!)=1$, if $n$ is even, and for every $\ZZ{m}$ such that $\gcd\left(m,\left(\frac{3n+1}{2}\right)!\right)=1$, if $n$ is odd.
\end{corollary}

\begin{proof}
For the Pascal automaton of dimension $n-1$, we have $\sigma=n$ and $\sigma_k=-1$ for all $k\in[1,n-1]$. Let $m$ be a positive integer such that $\gcd(m,n!)=1$ and let $A=\ArA(a,d)$ be the arithmetic array of $\ZZ{m}$ with common difference $d=(d_1,\ldots,d_{n-1})$ defined by $d_k:=k\in\ZZ{m}$ for all $k\in[1,n-1]$. Then,
$$
d_n := \sigma^{-1}\sum_{k=1}^{n-1}\sigma_kd_k = -n^{-1}\sum_{k=1}^{n-1}k = -2^{-1}(n-1).
$$
Since $\gcd(m,n!)=1$, the common differences $d_1,d_2,\ldots,d_{n-1}$ and $d_n$ are invertible in $\ZZ{m}$. For all distinct integers $u,v\in[1,n-1]$, we have
\begin{equation}\label{eq41}
3 \le u + v\le 2n-3 \quad\text{and}\quad 1\le |v-u| \le n-2.
\end{equation}
For all integers $u\in[1,n-1]$, we have
\begin{equation}\label{eq42}
n+1 \le 2u+(n-1)\le 3(n-1) \quad \text{and}\quad  -(n-3)\le 2u-(n-1) \le n-1.
\end{equation}
If $n$ is even, then $2u-(n-1)$ cannot vanish and we deduce that if $\gcd(m,(3(n-1))!)=1$ then, for any orientation $\varepsilon\in\{-1,1\}^n$, all the elements in
\begin{equation}\label{eq43}
\left\{\varepsilon_v d_v -\varepsilon_u d_u\ \middle|\ 1\le u<v\le n \right\}
\end{equation}
are invertible modulo $m$. Therefore, by Theorem~\ref{thm1}, all simplices of orientation $\varepsilon$ and of size $s\equiv -t\bmod{\lcm(\ord_m(n),m)}$, with $t\in[0,n-1]$, appearing in the orbit $\orb(A)$ are balanced in $\ZZ{m}$. For the specific orientations $\varepsilon=+\cdots+-$ and $\varepsilon=-\cdots-+$, we deduce from the second inequalities of (\ref{eq41}) and (\ref{eq42}) that this result is also true in the more general case where $\gcd(m,n!)=1$. Now, suppose that $n$ is odd and let $\varepsilon\in\{-1,1\}^{n}$ be an orientation. If there exists $l\in[1,n-1]$ such that $\varepsilon_l=\varepsilon_n$ , then we consider the arithmetic array $A=\ArA(a,d)$ of $\ZZ{m}$ with common differences $d=(d_1,\ldots,d_{n-1})$ defined by $d_l:=\frac{n-1}{2}$, $d_\frac{n-1}{2}:=l$ and $d_k:=k$ for all $k\in[1,n-1]\setminus\left\{l,\frac{n-1}{2}\right\}$. Then, $d_n=-2^{-1}(n-1)$ as before and
$$
\varepsilon_ld_l - \varepsilon_n d_n = \varepsilon_l(n-1).
$$
If $\gcd(m,(3(n-1))!)=1$, we deduce that all the common differences $d_k$ for $k\in[1,n]$ and, from inequalities (\ref{eq41}) and (\ref{eq42}), all the elements in the set of (\ref{eq43}) are invertible modulo $m$. Therefore, by Theorem~\ref{thm1}, all simplices of this orientation $\varepsilon$, where $\varepsilon_l=\varepsilon_n$, and of size $s\equiv -t\bmod{\lcm(\ord_m(n),m)}$, with $t\in[0,n-1]$, appearing in the orbit $\orb(A)$ are balanced in $\ZZ{m}$. Finally, suppose that $n$ is odd and $\varepsilon$ is such that $\varepsilon_k=-\varepsilon_n$ for all $k\in[1,n-1]$, i.e., $\varepsilon=+\cdots+-$ or $\varepsilon=-\cdots-+$. In this case, we consider the arithmetic array $A=\ArA(a,d)$ of $\ZZ{m}$ with common differences $d=(d_1,\ldots,d_{n-1})$ defined by $d_{\frac{n+1}{2}}:=\frac{3n+1}{2}$ and $d_k:=k$ for all $k\in[1,n-1]\setminus\left\{\frac{n+1}{2}\right\}$. Then, $d_n=-2^{-1}(n+1)$. For all distinct integers $u,v\in\left[1,\frac{3n+1}{2}\right]$, we have
\begin{equation}\label{eq44}
1\le |v-u|\le \frac{3n-1}{2} \quad\text{and}\quad -\frac{n-1}{2}\le u-\frac{n+1}{2} \le n.
\end{equation}
If $\ZZ{m}$ is such that $\gcd\left(m,\left(\frac{3n+1}{2}\right)!\right)=1$, we know by definition that all the common differences $d_k$ are invertible modulo $m$ for all $k\in[1,n]$. Moreover, since $d_u+d_n$ cannot vanish by definition, and from (\ref{eq44}), we deduce that all the elements in the set of (\ref{eq43}) are invertible modulo $m$. By Theorem~\ref{thm1} again, all simplices with these orientations, $\varepsilon=+\cdots+-$ or $\varepsilon=-\cdots-+$, and of size $s\equiv -t\bmod{\lcm(\ord_m(n),m)}$, with $t\in[0,n-1]$, appearing in the orbit $\orb(A)$ are balanced in $\ZZ{m}$. This concludes the proof.
\end{proof}

\subsection{In dimension 3}

In this subsection, we show that, in dimension 3, a similar result than Theorem~\ref{thm1} can be obtained for certain even values of $m$ by using Theorem~\ref{thm5}. As a corollary, the special case of the Pascal cellular automaton $\PCA{2}$ is studied.

\subsubsection{For any ACA}

\begin{theorem}\label{thm6}
Let $m$ be an even number not divisible by $3$ such that $\sigma\in\ZZ{m}$ is invertible and $\sigma\equiv 1 \bmod{2^{v_2(m)}}$, where $v_2(m)$ is the highest exponent $u$ such that $2^u$ divides $m$. Let $a\in\ZZ{m}$, $d=(d_1,d_2)\in\left(\ZZ{m}\right)^{2}$ and $\varepsilon\in\{-1,1\}^{3}$ such that $\varepsilon_2d_2$, $\varepsilon_3d_3$, $\varepsilon_2d_2-\varepsilon_1d_1$, $\varepsilon_1d_1-\varepsilon_3d_3$ are invertible in $\ZZ{m}$ and $\gcd(\varepsilon_1d_1,m) = \gcd(\varepsilon_3d_3-\varepsilon_2d_2,m) = 2$, where $d_3:=\sigma^{-1}(\sigma_1d_1+\sigma_2d_2)$. Then, in the orbit $\orb(\ArA(a,d))$, every tetrahedron with orientation $\varepsilon$ and of size $s$ is balanced, for all $s\equiv 0$ or $-2\bmod{\lcm(\ord_m(\sigma),m)}$.
\end{theorem}

\begin{proof}
Let $\triangle=\triangle(j,\varepsilon,s)$ be a tetrahedron of size $\lambda\lcm(\ord_m(\sigma),m)-t$, where $t\in\{0,2\}$, appearing in the orbit $\mathcal{O}(\ArA(a,d))=(a_i)_{i\in\Z^{2}\times\N}$. We proceed by induction on $m$.

For $m=2^{v_2(m)}$, since $\sigma\equiv 1\bmod{m}$, it follows that $\triangle$ is an arithmetic tetrahedron of size $s\equiv 0$ or $-2\bmod{m}$ and of common differences $(\varepsilon_1d_1,\varepsilon_2d_2,\varepsilon_3d_3)$ such that $\varepsilon_2d_2$, $\varepsilon_2d_2-\varepsilon_1d_1$, $\varepsilon_3d_3$ and $\varepsilon_1d_1-\varepsilon_3d_3$ are invertible and $\gcd(\varepsilon_1d_1,m)=\gcd(\varepsilon_3d_3- \varepsilon_2d_2)=2$. By Theorem~\ref{thm5}, the tetrahedron $\triangle$ is balanced in this case.

Suppose now that $m>2^{v_2(m)}$ and let $m':=m/2^{v_2(m)}$, the odd part of $m$. Let
$$
m_1:=2^{v_2(m)}\gcd(\ord_m(\sigma),m')\quad\text{and}\quad m_2:=\frac{m}{m_1}=\frac{m'}{\gcd(\ord_m(\sigma),m')}.
$$
First, we prove that $\m_{\triangle}(x+m_1)=\m_{\triangle}(x)$, for all $x\in\ZZ{m}$. Since
$$
m_1 = 2^{v_2(m)}\gcd(\ord_m(\sigma),m') = \gcd(2^{v_2(m)}\ord_m(\sigma),m),
$$
it follows that $m_1$ is divisible by $\gcd(\ord_m(\sigma),m)$. By Proposition~\ref{propsub} with
$$
\alpha=\frac{m_1}{\gcd(\ord_m(\sigma),m)}\ord_m(\sigma),
$$
since $s=\alpha m_2-t$, we know that $\triangle$ can be decomposed into $\alpha^{3}$ subtetrahedra that are arithmetic. More precisely, we have
$$
\triangle = \bigcup_{k\in[0,\alpha-1]^{3}}\SuS_k,
$$
where
$$
\SuS_k = \AS\left(a_{j+\varepsilon\cdot k},m_1\frac{\ord_m(\sigma)}{\gcd(\ord_m(\sigma),m)}\sigma^{j_3+\varepsilon_3 k_3}(\varepsilon_1d_1,\varepsilon_2d_2,\varepsilon_3d_3) , m_2-\left\lfloor\ \frac{k_1+k_2+k_3}{\alpha} \right\rfloor\right),
$$
for all $k\in[0,\alpha-1]^{3}$. Since $m_2$ is an odd factor of $m$ and $\gcd(\varepsilon_2d_2,m)=\gcd(\varepsilon_3d_3-\varepsilon_2d_2,m)=2$, we deduce that $\pi_{m_2}(\varepsilon_2d_2)$ and $\pi_{m_2}(\varepsilon_3d_3-\varepsilon_2d_2)$ are invertible in $\ZZ{m_2}$. Since $\sigma$, $\varepsilon_2d_2$, $\varepsilon_3d_3$, $\varepsilon_2d_2-\varepsilon_1d_1$ and $\varepsilon_1d_1-\varepsilon_3d_3$ are invertible in $\ZZ{m}$ and since
$$
\gcd\left(\frac{\ord_m(\sigma)}{\gcd(\ord_m(\sigma),m)},m_2\right) \mid \gcd\left(\frac{\ord_m(\sigma)}{\gcd(\ord_m(\sigma),m')},\frac{m'}{\gcd(\ord_m(\sigma),m')}\right) = 1,
$$
we obtain that the elements
$$
\pi_{m_2}\left( \frac{\ord_m(\sigma)}{\gcd(\ord_m(\sigma),m)}\sigma^{j_3+\varepsilon_3 k_3}\varepsilon_ud_u  \right),
$$
for all $u\in[1,3]$, and
$$
\pi_{m_2}\left( \frac{\ord_m(\sigma)}{\gcd(\ord_m(\sigma),m)}\sigma^{j_3+\varepsilon_3 k_3}(\varepsilon_vd_v-\varepsilon_ud_u)  \right),
$$
for all distinct integers $u,v\in[1,3]$, are invertible in $\ZZ{m_2}$. Moreover, since the size of $\SuS_k$ is $m_2-\left\lfloor\ \frac{k_1+k_2+k_3}{\alpha} \right\rfloor$, which is congruent to $0$, $-1$, or $-2$ modulo $m_2$, we deduce from Lemma~\ref{lem1} that the multiplicity function of $\SuS_k$ verifies
$$
\m_{\SuS_k}\left(x+m_1\right)=\m_{\SuS_k}(x),
$$
for all $x\in\ZZ{m}$ and for all $k\in[0,\alpha-1]^{3}$. Therefore, we have
$$
\m_{\triangle}\left(x+m_1\right) = \displaystyle\sum_{k\in[0,\alpha-1]^{3}}\m_{\SuS_k}\left(x+m_1\right) = \displaystyle\sum_{k\in[0,\alpha-1]^{3}}\m_{\SuS_k}\left(x\right) = m_{\triangle}\left(x\right),
$$
for all $x\in\ZZ{m}$. Since $\sigma\equiv1\bmod{2^{v_2(m)}}$, it follows that $\ord_m(\sigma)=\ord_{m'}(\sigma)$ and thus $m_1=2^{v_2(m)}\gcd(\ord_m(\sigma),m')=2^{v_2(m)}\gcd(\ord_{m'}(\sigma),m')$ is a proper divisor of $m$. Thus, from the induction hypothesis, we know that the projection of $\triangle$ into $\ZZ{m_1}$ is balanced. Finally, since $\m_{\triangle}(x+m_1)=\m_{\triangle}(x)$ for all $x\in\ZZ{m}$ and $\pi_{m_1}(\triangle)$ balanced in $\ZZ{m_1}$, we conclude that the tetrahedron $\triangle$ is balanced in $\ZZ{m}$ by Theorem~\ref{thmproj}. This completes the proof.
\end{proof}

\subsubsection{For the Pascal cellular automaton}

Here, we investigate the consequences of Theorem~\ref{thm6} on the existence of balanced tetrahedra, in the case where the ACA considered is $\PCA{2}$.

\begin{corollary}
For even numbers $m$ not divisible by $3$ such that $v_2(m)=1$, there exist infinitely many balanced tetrahedra of $\ZZ{m}$ generated by $\PCA{2}$, for all orientations $\varepsilon = +++$, $+-+$, $+--$, $-++$, $-+-$ and $---$. In the remaining case of the two orientations $\varepsilon=++-$ or $\varepsilon=--+$, the existence of an infinite number of such balanced tetrahedra is verified for every $\ZZ{m}$ of even order $m$ such that $v_2(m)=1$ and $\gcd(m,3.5)=1$.
\end{corollary}

\begin{proof}
For the Pascal automaton of dimension $2$, we have $\sigma=3$ and $\sigma_1=\sigma_2=-1$. If there exists $i\in\{1,2\}$ such that $\varepsilon_i=\varepsilon_3$, then we consider the orbit associated with the arithmetic array $A=\ArA(a,(d_1,d_2))$ of $\ZZ{m}$, where $d_i:=1$ and $d_j:=2$ with $\{i,j\}=\{1,2\}$. Then, $d_3=-3^{-1}.3=-1$, $d_i+d_j=3$, $d_j-d_i=1$, $d_j+d_3=1$, $d_j-d_3=3$ and $d_i-d_3=2$. It follows that, if $m$ is an even number not divisible by $3$ and such that $v_2(m)=1$, we have
$$
\gcd(\varepsilon_id_i,m) = \gcd(\varepsilon_3d_3,m) = \gcd(\varepsilon_jd_j-\varepsilon_id_i,m) = \gcd(\varepsilon_jd_j-\varepsilon_3d_3,m) = 1
$$
and
$$
\gcd(\varepsilon_jd_j,m) = \gcd(\varepsilon_id_i-\varepsilon_3d_3,m)=2.
$$
Therefore, by Theorem~\ref{thm6}, any tetrahedron of this orientation $\varepsilon$, where $\varepsilon_i=\varepsilon_3$, and of size $s\equiv 0$ or $-2\bmod{\lcm(\ord_m(3),m)}$ appearing in the orbit $\orb(A)$ is balanced in $\ZZ{m}$. Suppose now that $\varepsilon=++-$ or $--+$ and consider the orbit associated with the arithmetic array $A=\ArA(a,(d_1,d_2))$ of $\ZZ{m}$, where $d_1:=4$ and $d_2:=5$. Then, $d_3=-3^{-1}.9=-3$, $d_2-d_1=1$, $d_1+d_3=1$ and $d_2+d_3=2$. It follows that, if $m$ is an even number such that $v_2(m)=1$ and $\gcd(m,3.5)=1$, we have
$$
\gcd(\varepsilon_2d_2,m) = \gcd(\varepsilon_3d_3,m) = \gcd(\varepsilon_2d_2-\varepsilon_1d_1,m) = \gcd(\varepsilon_1d_1-\varepsilon_3d_3,m) = 1
$$
and
$$
\gcd(\varepsilon_1d_1,m) = \gcd(\varepsilon_2d_2-\varepsilon_3d_3,m)=2.
$$
Therefore, by Theorem~\ref{thm6}, any tetrahedron of the orientation $\varepsilon=++-$ or $--+$ and of size $s\equiv 0$ or $-2\bmod{\lcm(\ord_m(3),m)}$ appearing in the orbit $\orb(A)$ is balanced in $\ZZ{m}$. This concludes the proof.
\end{proof}

\section{The antisymmetric case}

We begin this section by defining the antisymmetric sequences and the antisymmetric simplices.

\begin{definition}[Antisymmetric sequences]
A finite sequence $S=(a_1,\ldots,a_{s})$ of length $s\ge 1$ in $\ZZ{m}$ is said to be \textit{antisymmetric} if $a_i+a_{s-i+1}=0$ for all $i\in[1,s]$.
\end{definition}

For instance, the sequence $S=(2,2,1,0,4,3,3)$ is antisymmetric in $\ZZ{5}$.

\begin{definition}[Antisymmetric simplices]
Let $A=\left(a_{i}\right)_{i\in\Z^{n}}$ be an infinite array of elements in $\ZZ{m}$ and let $\triangle(j,\varepsilon,s)$ be the $n$-simplex of size $s$, with orientation $\varepsilon\in\{-1,1\}^{n}$ and whose principal vertex is $a_j$ in $A$, that is,
$$
\triangle(j,\varepsilon,s) = \left\{ a_{j+\varepsilon\cdot k} \ \middle|\ k\in\N^{n}\ \text{such that}\ k_1+\cdots+k_n\le s-1\right\}.
$$
Let $u$ and $v$ be two distinct integers in $[0,n]$. The simplex $\triangle(j,\varepsilon,s)$ is said to be \textit{$(u,v)$-antisymmetric} if all its subsequences in the same direction of the edge $\E_{u,v}$ between the vertices $\V_{u}$ and $\V_{v}$ are antisymmetric. More precisely, $\triangle(j,\varepsilon,s)$ is $(0,v)$-antisymmetric if we have
$$
a_{j+\varepsilon\cdot k} + a_{j+\varepsilon\cdot s(k)} = 0,\text{ where } s(k) = \left(k_1,\ldots,k_{v-1},s-1-\sum_{l=1}^{n}k_l,k_{v+1},\ldots,k_n\right),
$$
for all $k\in\N^n$ such that $k_1+\cdots+k_n\le s-1$ and, for $u,v\ge1$, $\triangle(j,\varepsilon,s)$ is $(u,v)$-antisymmetric if we have
$$
a_{j+\varepsilon\cdot k} + a_{j+\varepsilon\cdot t(k)} = 0,\ \text{where}\ (t(k))_l = k_{\tau(l)}\ \text{for all}\ l\in[1,n],
$$
where $\tau$ is the transposition $(u,v)$, for all $k\in\N^n$ such that $k_1+\cdots+k_n\le s-1$.
\end{definition}

For instance, the tetrahedron depicted in Figure~\ref{fig9} is $(1,2)$-antisymmetric. Moreover, each row of this tetrahedron is an $(1,2)$-antisymmetric triangle.

\begin{figure}
\begin{center}
\begin{tikzpicture}[scale=0.5]

\draw (1,1) -- (1,6);
\draw (2,1) -- (2,6);
\draw (3,2) -- (3,6);
\draw (4,3) -- (4,6);
\draw (5,4) -- (5,6);
\draw (6,5) -- (6,6);
\draw (1,6) -- (6,6);
\draw (1,5) -- (6,5);
\draw (1,4) -- (5,4);
\draw (1,3) -- (4,3);
\draw (1,2) -- (3,2);
\draw (1,1) -- (2,1);

\node at (1.5,5.5) {$0$};
\node at (2.5,5.5) {$1$};
\node at (3.5,5.5) {$1$};
\node at (4.5,5.5) {$3$};
\node at (5.5,5.5) {$6$};
\node at (1.5,4.5) {$6$};
\node at (2.5,4.5) {$0$};
\node at (3.5,4.5) {$4$};
\node at (4.5,4.5) {$5$};
\node at (1.5,3.5) {$6$};
\node at (2.5,3.5) {$3$};
\node at (3.5,3.5) {$0$};
\node at (1.5,2.5) {$4$};
\node at (2.5,2.5) {$2$};
\node at (1.5,1.5) {$1$};

\draw (7,2) -- (7,6);
\draw (8,2) -- (8,6);
\draw (9,3) -- (9,6);
\draw (10,4) -- (10,6);
\draw (11,5) -- (11,6);
\draw (7,6) -- (11,6);
\draw (7,5) -- (11,5);
\draw (7,4) -- (10,4);
\draw (7,3) -- (9,3);
\draw (7,2) -- (8,2);

\node at (7.5,5.5) {$0$};
\node at (8.5,5.5) {$4$};
\node at (9.5,5.5) {$0$};
\node at (10.5,5.5) {$1$};
\node at (7.5,4.5) {$3$};
\node at (8.5,4.5) {$0$};
\node at (9.5,4.5) {$3$};
\node at (7.5,3.5) {$0$};
\node at (8.5,3.5) {$4$};
\node at (7.5,2.5) {$6$};

\draw (12,3) -- (12,6);
\draw (13,3) -- (13,6);
\draw (14,4) -- (14,6);
\draw (15,5) -- (15,6);
\draw (12,6) -- (15,6);
\draw (12,5) -- (15,5);
\draw (12,4) -- (14,4);
\draw (12,3) -- (13,3);

\node at (12.5,5.5) {$0$};
\node at (13.5,5.5) {$2$};
\node at (14.5,5.5) {$3$};
\node at (12.5,4.5) {$5$};
\node at (13.5,4.5) {$0$};
\node at (12.5,3.5) {$4$};

\draw (16,4) -- (16,6);
\draw (17,4) -- (17,6);
\draw (18,5) -- (18,6);
\draw (16,6) -- (18,6);
\draw (16,5) -- (18,5);
\draw (16,4) -- (17,4);

\node at (16.5,5.5) {$0$};
\node at (17.5,5.5) {$5$};
\node at (16.5,4.5) {$2$};

\draw (19,6) -- (20,6) -- (20,5) -- (19,5) -- (19,6);

\node at (19.5,5.5) {$0$};

\end{tikzpicture}
\end{center}
\caption{An $(1,2)$-antisymmetric tetrahedron in $\ZZ{7}$}\label{fig9}
\end{figure}

\subsection{Antisymmetric simplices}

Let $m$ and $n$ be two positive integers such that $n\ge 2$ and $\gcd(m,n!)=1$. In the sequel of this section, we consider $n$-simplices $\triangle(j,\varepsilon,s)$ appearing in the orbit of the arithmetic array $\ArA(a,d)$, where $a\in\ZZ{m}$ and $d=(d_1,d_2,\ldots,d_{n-1})\in\left(\ZZ{m}\right)^{n-1}$, generated by an ACA of weight array $W=(w_l)_{l\in[-r,r]^{n-1}}$. The elements of this orbit are denoted by $\orb(\ArA(a,d))=(a_i)_{i\in\Z^{n-1}\times\N}$. As already defined before,
$$
\sigma := \sum_{l\in[-r,r]^{n-1}}w_l\quad\text{and}\quad \sigma_k = \sum_{l\in[-r,r]^{n-1}}l_k w_{l},\ \text{for all }k\in[1,n-1].
$$
Moreover, suppose that $\sigma$ is invertible modulo $m$ and let
$$
d_n: = \sigma^{-1}\sum_{k=1}^{n-1}\sigma_k d_k.
$$
In this subsection, necessary conditions on simplices for being antisymmetric are determined.

\begin{proposition}\label{prop13}
Let $u$ and $v$ be two distinct integers in $[1,n]$. If $\triangle(j,\varepsilon,s)$ is $(u,v)$-antisymmetric, then $d_w=0$ for all $w\in[1,n]\setminus\{u,v\}$ and $\varepsilon_{u}d_{u}+\varepsilon_{v}d_{v} = 0$.
\end{proposition}

\begin{proof}
Let $k\in\N^{n}$ such that $k_1+\cdots+k_n\le s-1$. If $k_{u}=k_{v}$, then $\tau(k)=k$ and
$$
2a_{j+\varepsilon \cdot k} = a_{j+\varepsilon \cdot k}+a_{j+\varepsilon\cdot\tau(k)} = 0,
$$
by definition of the $(u,v)$-antisymmetry. It follows that $a_{j+\varepsilon\cdot k}=0$ in this case. Therefore, if $k_u=k_v$, we have
$$
\sigma^{j_n+\varepsilon_n k_n}\left(a+\sum_{l=1}^{n}(j_l+\varepsilon_l k_l)d_l\right) = 0 \ \Longleftrightarrow\ a+\sum_{l=1}^{n}(j_l+\varepsilon_l k_l)d_l = 0.
$$
For $k=0$, we obtain that
$$
a + \sum_{l=1}^{n}j_l d_l = 0.
$$
Now, we consider the canonical basis $(e_1,\ldots,e_n)$ of the vector space $\Z^n$. For all $w\in[1,n]\setminus\{u,v\}$, since $k_{u}=k_{v}=0$ in $k=e_w$, we have
$$
a_{j+\varepsilon\cdot e_w}=0\ \Longleftrightarrow\  a+\sum_{l=1}^{n}j_ld_l + \varepsilon_w d_w = 0\ \Longleftrightarrow\ \varepsilon_w d_w=0 \ \Longleftrightarrow\ d_w=0.
$$
Finally, for $k=e_u+e_v$, since $k_u=k_v=1$, we obtain that
$$
a_{j+\varepsilon\cdot k}=0\ \Longleftrightarrow\ a + \sum_{l=1}^{n}j_ld_l + \varepsilon_{u}d_{u} + \varepsilon_{v}d_{v} = 0\ \Longleftrightarrow\  \varepsilon_{u}d_{u} + \varepsilon_{v}d_{v} = 0.
$$
This completes the proof.
\end{proof}

\begin{proposition}\label{prop14}
Let $v$ be an integer in $[1,n]$. If $\triangle(j,\varepsilon,s)$ is $(0,v)$-antisymmetric, then $2\varepsilon_wd_w=\varepsilon_v d_v$ for all $w\in[1,n]\setminus\{v\}$.
\end{proposition}

\begin{proof}
Let $k\in\N^{n}$ such that $k_1+\cdots+k_n\le s-1$. If $k_{v}=s-1-\sum_{l=1}^{n}k_l$, then $s(k)=k$ and
$$
2a_{j+\varepsilon\cdot k} = a_{j+\varepsilon\cdot k}+a_{j+\varepsilon\cdot s(k)} = 0,
$$
by definition of the $(0,v)$-antisymmetry. It follows that $a_{j+\varepsilon\cdot k}=0$ in this case. Therefore, if $k_{v}=s-1-\sum_{l=1}^{n}k_l$, we have
$$
\sigma^{j_n+\varepsilon_nk_n}\left(a+\sum_{l=1}^{n}(j_l+\varepsilon_lk_l)d_l\right) = 0 \ \Longleftrightarrow\ a+\sum_{l=1}^{n}(j_l+\varepsilon_lk_l)d_l = 0.
$$
Let $w\in[1,n]\setminus\{v\}$. Since $k_{v}+\sum_{l=1}^{n}k_l=s-1$ for $k=(s-1)e_w$, we have
$$
a_{j+\varepsilon\cdot (s-1)e_w}=0\ \Longleftrightarrow\ a+\sum_{l=1}^{n}j_ld_l + \varepsilon_w(s-1)d_w=0.
$$
Moreover, for $k=e_{v}+(s-3)e_{w}$, since $\sum_{l=1}^{n}k_l+k_{v}=s-1$, we obtain that
$$
a_{j+\varepsilon\cdot k}=0 \Longleftrightarrow a + \displaystyle\sum_{l=1}^{n}j_ld_l + \varepsilon_{v}d_{v} + \varepsilon_{w}(s-3)d_{w} = 0.
$$
It follows that
$$
\varepsilon_{v}d_{v} + \varepsilon_{w}(s-3)d_{w} = \varepsilon_w(s-1)d_w \Longleftrightarrow \varepsilon_{v}d_{v} = 2\varepsilon_w d_w.
$$
This completes the proof.
\end{proof}

For $n\ge3$, it is easy to see from Proposition~\ref{prop13} and Proposition~\ref{prop14} that if the simplex $\triangle(j,\varepsilon,s)$ is $(u,v)$-antisymmetric, then there is at least one element among the elements $\varepsilon_id_i$, for all $i\in[1,n]$, and $\varepsilon_jd_j-\varepsilon_id_i$, for all distinct integers $i,j\in[1,n]$, which is non-invertible and equal to zero in $\ZZ{m}$. In this case, the hypotheses of Theorem~\ref{thm2} are not satisfied. Therefore, in the next subsection, we only consider the case of dimension $n=2$.

\subsection{In dimension 2}

An arithmetic array of dimension $1$ is simply called an \textit{arithmetic progression} and is denoted by $\mathrm{AP}(a,d)$ or $\mathrm{AP}(a,d,s)$, for an arithmetic progression of length $s$, that is,
$$
\AP(a,d,s) = (a,a+d,a+2d,\ldots,a+(s-1)d).
$$
Let $W=(w_{-r},\ldots,w_r)\in\Z^{2r+1}$ be the weight sequence of the ACA of dimension $1$ that we consider here. By Lemma~\ref{prop2}, we know that the derived sequence of an arithmetic progression is also an arithmetic progression. Indeed, we have
$$
\partial\AP(a,d) = \AP(\sigma a+\sigma'd,\sigma d),
$$
where
$$
\sigma = \sum_{i=-r}^{r}w_i \quad\text{and}\quad \sigma' = \sum_{i=-r}^{r}i w_i.
$$
As already remarked, for $\overline{W}=(0,\sigma-\sigma',\sigma')$, we obtain that $\sigma\left(\overline{W}\right) = \sigma(W)$ and $\sigma'\left(\overline{W}\right) = \sigma'(W)$. Thus, the orbits of $\AP(a,d)$ are the same if we consider $W$ or $\overline{W}$. Therefore, in the sequel of this subsection, we only consider the case where $r=1$ and $W=(0,\sigma-\sigma',\sigma')$.

\subsubsection{(1,2)-antisymmetric triangles}

First, we know from Proposition~\ref{prop13} that if the triangle $\triangle(j,\varepsilon,s)$, appearing in the orbit $\orb(\AP(a,d))$ with $d$ invertible, is $(1,2)$-antisymmetric, then
$$
\varepsilon_1d + \varepsilon_2\sigma^{-1}\sigma'd = 0\ \Longleftrightarrow\  \varepsilon_1\sigma+\varepsilon_2\sigma'=0\ \Longleftrightarrow\ \sigma' = -\varepsilon_1\varepsilon_2\sigma.
$$
So, we deduce that
$$
W = (0,(1+\varepsilon_1\varepsilon_2)\sigma,-\varepsilon_1\varepsilon_2\sigma).
$$
Since $\triangle(j,\varepsilon,s)$ is $(1,2)$-antisymmetric, we know that $a_j=0$ and $a_{j+\varepsilon\cdot e_1}+a_{j+\varepsilon\cdot e_2} = 0$. It follows that
$$
a_j=0 \ \Longleftrightarrow\ a+j_1d_1+j_2d_2 = 0,
$$
and
$$
a_{j+\varepsilon\cdot e_1} + a_{j+\varepsilon\cdot e_2} = 0
\begin{array}[t]{cl}
\Longleftrightarrow & \sigma^{j_2}(a+(j_1+\varepsilon_1)d_1 + j_2d_2) + \sigma^{j_2+\varepsilon_2}(a+j_1d_1 + (j_2+\varepsilon_2)d_2) = 0 \\
\Longrightarrow & \varepsilon_1d_1 + \sigma^{\varepsilon_2}\varepsilon_2d_2 = 0 \\
\Longleftrightarrow & \varepsilon_1 + \sigma^{\varepsilon_2}\varepsilon_2\sigma^{-1}\sigma' = 0.
\end{array}
$$
This implies that $\sigma^{\varepsilon_2}=1$ and thus $\sigma=1$. Therefore,
$$
W = (0,1+\varepsilon_1\varepsilon_2,-\varepsilon_1\varepsilon_2).
$$
Finally, since $\sigma=1$, we know that $\triangle(j,\varepsilon,s)$ is an arithmetic triangle which is already balanced for all $s\equiv 0$ or $-1\bmod{m}$ by Theorem~\ref{thm2}.

\subsubsection{(0,2)-antisymmetric triangles}

First, we know from Proposition~\ref{prop14} that if the triangle $\triangle(j,\varepsilon,s)$, appearing in the orbit $\orb(\AP(a,d))$ with $d$ invertible, is $(0,2)$-antisymmetric, then
$$
\varepsilon_2\sigma^{-1}\sigma' = 2\varepsilon_1\ \Longleftrightarrow\ \sigma' = 2\varepsilon_1\varepsilon_2\sigma.
$$
So, we deduce that
$$
W = (0,(1-2\varepsilon_1\varepsilon_2)\sigma,2\varepsilon_1\varepsilon_2\sigma).
$$
Since $\triangle(j,\varepsilon,s)$ is $(0,2)$-antisymmetric, we know that $a_{j+(s-1)\varepsilon\cdot e_1}=0$ and $a_{j+(s-2)\varepsilon\cdot e_1} + a_{j+\varepsilon\cdot((s-2)e_1+e_2)} = 0$. It follows that
$$
a_{j+(s-1)\varepsilon\cdot e_1}=0 \ \Longleftrightarrow\ a+(j_1+\varepsilon_1(s-1))d_1+j_2d_2 = 0,
$$
and
$$
a_{j+(s-2)\varepsilon\cdot e_1} + a_{j+\varepsilon\cdot((s-2)e_1+e_2)} = 0
\begin{array}[t]{l}
\Longleftrightarrow\quad \begin{array}[t]{l} \sigma^{j_2}(a+(j_1+\varepsilon_1(s-2))d_1 + j_2d_2) \\ + \sigma^{j_2+\varepsilon_2}(a+(j_1+\varepsilon_1(s-2))d_1 + (j_2+\varepsilon_2)d_2) = 0 \end{array}\\
\Longrightarrow\quad -\varepsilon_1d_1 + \sigma^{\varepsilon_2}(-\varepsilon_1d_1+\varepsilon_2d_2) = 0 \\
\Longleftrightarrow\quad -\varepsilon_1 + \sigma^{\varepsilon_2}(-\varepsilon_1+\varepsilon_2\sigma^{-1}\sigma') = 0\\
\Longleftrightarrow\quad \sigma' = \displaystyle\frac{1+\sigma^{\varepsilon_2}}{\sigma^{\varepsilon_2}}\varepsilon_1\varepsilon_2\sigma
\end{array}
$$
This implies that $\sigma^{\varepsilon_2}=1$ and thus $\sigma=1$. Therefore,
$$
W = (0,1-2\varepsilon_1\varepsilon_2,2\varepsilon_1\varepsilon_2).
$$
Finally, since $\sigma=1$, we know that $\triangle(j,\varepsilon,s)$ is an arithmetic triangle which is already balanced for all $s\equiv 0$ or $-1\bmod{m}$ by Theorem~\ref{thm2}.

\subsubsection{(0,1)-antisymmetric triangles}

First, we know from Proposition~\ref{prop14} that if the simplex $\triangle(j,\varepsilon,s)$ is $(0,1)$-antisymmetric, then
$$
\varepsilon_1 = 2\varepsilon_2\sigma^{-1}\sigma' \quad\Longleftrightarrow\quad \sigma = 2\varepsilon_1\varepsilon_2\sigma'.
$$
So, we deduce that
$$
W = (0,(2\varepsilon_1\varepsilon_2-1)\sigma',\sigma').
$$

Now, we refine Theorem~\ref{thm1} in this case by considering triangles that have the additional property to be $(0,1)$-antisymmetric.

\begin{theorem}\label{thm7}
Let $m$ be an odd positive integers and let $W\in\Z^{2r+1}$ such that $\sigma=2\sigma'$ and $\sigma$ is invertible modulo $m$. Let $a,d\in\ZZ{m}$ such that $d$ is invertible. Then, in the orbit $\orb\left(\AP(a,d)\right)$, every $(0,1)$-antisymmetric triangle of orientation $(++)$ or $(--)$ and of size $s$ is balanced, for all positive integers $s\equiv 0$ or $-1 \bmod \lcm(\pord_m(\sigma),m)$, where $\pord_m(\sigma)$ is the multiplicative order of $\sigma$ in $\left(\ZZ{m}\right)^{*}/\{-1,1\}$.
\end{theorem}

\begin{proof}
Let $\triangle=\triangle(j,\varepsilon,s)$ be a triangle of size $s$ and with orientation $\varepsilon=++$ or $--$ appearing in the orbit $\orb(\AP(a,d))=(a_{i})_{i\in\Z\times\N}$. It is clear that $\ord_m(\sigma)=\pord_m(\sigma)$ or $\ord_m(\sigma)=2\pord_m(\sigma)$. From Theorem~\ref{thm1}, we already know that $\triangle$ is balanced for $s\equiv 0$ or $-1 \bmod{\lcm(\ord_m(\sigma),m)}$. Suppose now that $\ord_m(\sigma)=2\pord_m(\sigma)$. Different cases for the congruence of $s$ modulo $\pord_m(\sigma)$ are distinguished.
\setcounter{case}{0}
\begin{case}
$s\equiv 0\bmod{\lcm(\pord_m(\sigma),m)}$.\\
We consider the triangle $\triangle'=\triangle(j,\varepsilon,2s)$. Since $2s\equiv 0\bmod{\lcm(\ord_m(\sigma),m)}$, the triangle $\triangle'$ is balanced by Theorem~\ref{thm1}. As depicted in Figure~\ref{fig8}, $\triangle'$ can be decomposed into two triangles of size $s$, that are $\triangle(j,\varepsilon,s)$ and $\triangle((j_1,j_2+\varepsilon_2s),\varepsilon,s)$, and $s$ arithmetic progressions of length $s$, that are
$$
S_k:=\left\{a_{j_1+\varepsilon_1 l,j_2+\varepsilon_2k}\ \middle| \ l\in[s-k,2s-k-1] \right\},
$$
for all $k\in[0,s-1]$. Since
$$
\partial^{j_2+\varepsilon_2k}\AP(a,d)=\AP(\sigma^{j_2+\varepsilon_2k}(a+(j_2+\varepsilon_2k)\sigma^{-1}\sigma'd),\sigma^{j_2+\varepsilon_2k}d),
$$
for all $k\in[0,s-1]$, by Proposition~\ref{cor1}, we deduce that $S_k$ is an arithmetic progression with invertible common difference and of length $s$, which is divisible by $m$. Therefore the sequence $S_k$ is balanced in $\ZZ{m}$ for all $k\in[0,s-1]$. Moreover, since $s\equiv 0\bmod{\lcm(\pord_m(\sigma),m)}$, we obtain that $\partial^{j_2+\varepsilon_2s}\AP(a,d)=\pm\partial^{j_2}\AP(a,d)$. If $\partial^{j_2+\varepsilon_2s}\AP(a,d)=\partial^{j_2}\AP(a,d)$, then $\triangle((j_1,j_2+\varepsilon_2s),\varepsilon,s)=\triangle(j,\varepsilon,s)$ and two copies of $\triangle$ can then be seen as the multiset difference of the balanced triangle $\triangle'$ and all the arithmetic progressions $S_k$, which are also balanced. Therefore, $\triangle$ is balanced in this case. Otherwise, if $\partial^{j_2+\varepsilon_2s}\AP(a,d)=-\partial^{j_2}\AP(a,d)$, then $\triangle((j_1,j_2+\varepsilon_2s),\varepsilon,s)$ is the opposite triangle of $\triangle$. Moreover, since $\triangle$ is antisymmetric,  it follows that
$$
\m_{\triangle((j_1,j_2+\varepsilon_2s),\varepsilon,s)}(x)=\m_{\triangle}(-x)=\m_{\triangle}(x)
$$
for all $x\in\ZZ{m}$. Finally, since $\triangle((j_1,j_2+\varepsilon_2s),\varepsilon,s)$ and $\triangle$ have the same multiplicity function and since they can be seen as the multiset difference of the balanced triangle $\triangle'$ and all the arithmetic progressions $S_k$, which are balanced, we deduce that $\triangle$ is also balanced in this case.
\end{case}
\begin{case}
$s\equiv -1\bmod{\lcm(\pord_m(\sigma),m)}$.\\
The triangle $\triangle(j,\varepsilon,s)$ can be seen as the multiset difference of $\triangle((j_1,j_2+\varepsilon_2),\varepsilon,s+1)$, which is balanced by Case~1, and the arithmetic progression $\AP(a_{j_1,j_2+\varepsilon_2},\sigma^{j_2+\varepsilon_2}d,s+1)$ of invertible common difference and of length $s+1\equiv 0\bmod{m}$, which is also balanced. The multiset difference of two balanced multisets is balanced. This completes the proof.
\end{case}
\end{proof}

\begin{figure}
\begin{center}
\begin{tikzpicture}[scale=0.4]
\draw (0,0) -- (0,8) -- (8,8) -- (0,0);
\draw (4,4) -- (0,4) -- (4,8);
\draw (1,5) -- (5,5);
\draw (2,6) -- (6,6);
\draw (3,7) -- (7,7);
\node at (4/3,20/3) {\small $\triangle_1$\ };
\node at (4/3,8/3) {\small $\triangle_2$};
\node at (11/2,7.5) {\small $S_0$};
\node at (9/2,6.5) {\small $S_1$};
\node at (7/2,5.5) {\small $S_2$};
\node at (5/2,4.5) {\small $S_3$};
\end{tikzpicture}
\begin{tikzpicture}[scale=0.4]
\draw (0,0) -- (0,8) -- (-8,0) -- (0,0);
\draw (-4,4) -- (0,4) -- (-4,0);
\draw (-5,3) -- (-1,3);
\draw (-6,2) -- (-2,2);
\draw (-7,1) -- (-3,1);
\node at (-4/3,16/3) {\small $\triangle_2$\ };
\node at (-4/3,4/3) {\small $\triangle_1$};
\node at (-11/2,0.5) {\small $S_0$};
\node at (-9/2,1.5) {\small $S_1$};
\node at (-7/2,2.5) {\small $S_2$};
\node at (-5/2,3.5) {\small $S_3$};
\end{tikzpicture}
\end{center}
\caption{Decomposition of $\triangle'=\triangle(j,\varepsilon,2s)$}\label{fig8}
\end{figure}

\begin{theorem}
Let $m$ be an odd positive integers and let $W\in\Z^{2r+1}$ such that $\sigma=-2\sigma'$ and $\sigma$ is invertible modulo $m$. Let $a,d\in\ZZ{m}$ such that $d$ is invertible. Then, in the orbit $\orb\left(\AP(a,d)\right)$, every $(0,1)$-antisymmetric triangle of orientation $(-+)$ or $(+-)$ and of size $s$ is balanced, for all positive integers $s\equiv 0$ or $-1 \bmod \lcm(\pord_m(\sigma),m)$, where $\pord_m(\sigma)$ is the multiplicative order of $\sigma$ in $\left(\ZZ{m}\right)^{*}/\{-1,1\}$.
\end{theorem}

\begin{proof}
Similar to the proof of Theorem~\ref{thm7}
\end{proof}

\section{Open problems}

For the Pascal cellular automaton of dimension $1$, the following problems remain open.

\begin{problem}
For $m$ even, do there exist infinitely many balanced triangles of $\ZZ{m}$, with any orientation, generated by $\PCA{1}$?
\end{problem}

\begin{problem}
For $m$ odd divisible by $3$, do there exist infinitely many balanced triangles of $\ZZ{m}$, with orientations $++$ and $--$, generated by $\PCA{1}$?
\end{problem}

For the Pascal cellular automaton of dimension $2$, the following problems remain open.

\begin{problem}
For $m$ divisible by $3$, do there exist infinitely many balanced tetrahedra of $\ZZ{m}$, with any orientation, generated by $\PCA{2}$?
\end{problem}

\begin{problem}
For $m$ even such that $v_2(m)\ge 2$, do there exist infinitely many balanced tetrahedra of $\ZZ{m}$, with any orientation, generated by $\PCA{2}$?
\end{problem}

\begin{problem}
For $m$ odd divisible by $5$, do there exist infinitely many balanced tetrahedra of $\ZZ{m}$, with orientations $++-$, $+-+$, $+--$, $-++$, $-+-$ and $---$, generated by $\PCA{2}$?
\end{problem}

\begin{problem}
For $m$ divisible by $5$ such that $v_2(m)\le 1$, do there exist infinitely many balanced tetrahedra of $\ZZ{m}$, with orientations $++-$ and $--+$, generated by $\PCA{2}$?
\end{problem}

\section*{Acknowledgments}
The author would like to thank the anonymous referee for the time spent reading this manuscript and for useful comments and remarks, which improved the presentation of the paper.


\end{document}